\documentclass[leqno,11pt]{amsart}

\usepackage{amssymb,mathtools,amsthm}
\usepackage[latin1]{inputenc}
\usepackage{stackrel}
\usepackage[toc,page]{appendix}
\usepackage[english]{babel}
\usepackage[colorlinks, citecolor=blue]{hyperref}
\usepackage{enumerate, enumitem}
\usepackage{tikz}

\newcommand{\R}{\mathbb R}
\newcommand{\N}{\mathbb N}
\newcommand{\Z}{\mathbb Z}

\newcommand{\al}{\alpha}

\newcommand{\ve}{\varepsilon}
\newcommand{\vt}{\vartheta}
\newcommand{\vp}{\varphi}
\newcommand{\cerchio}{\mathbb{S}^{1}} 

\newcommand{\uguale}{\stackrel{.}{=}} 

\theoremstyle{plain}
\newtheorem{theorem}{Theorem}[section]
\newtheorem{lemma}[theorem]{Lemma}
\newtheorem{proposition}[theorem]{Proposition}
\newtheorem{corollary}[theorem]{Corollary}

\theoremstyle{definition}

\newtheorem{definition}[theorem]{Definition}
\newtheorem{remark}[theorem]{Remark}

\title{Scattering Dynamics and Chaotic Motions in a Relativistic Two-Centre Problem}
\author{Stefano Baranzini \and Gian Marco Canneori}

\date{}

\usepackage[top=3cm, bottom=3cm, left=2cm, right=2cm]{geometry}

\keywords{$N$-centre problem, chaos, symbolic dynamics, variational methods, relativistic dynamics}
\subjclass[2020] {
	70F10, 
	34C28, 
	70G75, 
	37B10, 
	37N05, 
	70H40 
}

\begin{document}
\begin{abstract}
We study the planar relativistic two-centre problem at fixed energy, including a class of perturbations of the Keplerian potential. Up to a reparametrisation of time, the relativistic dynamics is equivalent to a classical two-centre system with critical strong-force singularities. This reduction allows us to apply variational methods based on the Maupertuis functional.

We construct collision-free relativistic trajectories with prescribed symbolic itineraries and obtain a coding of the dynamics by admissible sequences. In particular, we prove the existence of bounded orbits, scattering solutions with prescribed asymptotic directions, and trapped trajectories that are asymptotically free in one time direction and exhibit prescribed symbolic behaviour in the other.
\end{abstract}

\maketitle


\section{Introduction}

	In mechanics, the classical planar $N$-centre problem describes the motion of a body $q$ of mass $m$ on the plane in the gravitational field generated by $N$ heavy bodies of mass 
	$m_i$ fixed at positions $c_i \in \mathbb{R}^2$ and the equation of motion reads
	\begin{equation}\label{eq-kep}
		m \ddot q = - \sum_{i=1}^N m_i \frac{q-c_i}{\vert q-c_i \vert^3}, \qquad q\in \mathbb{R}^2 \setminus \{c_1,\dots,c_n\}.
	\end{equation} 
	It is well known that, in several energy regimes, if $N\ge 3$, the dynamics of this system is highly complex. See for instance  \cite{SoaTer2012,BolKoz2017_2,BarCan,BarCanTer2021,ChenYu2020Heteroclinic}and \cite{Yu2016Periodic},  where more general type of interactions are also considered.
	
	However, the $N$-centre model cannot be used in practice to describe the motion of a comet or a satellite under the gravitational influence of $N$ massive bodies, since it ignores many contributions coming from their interaction. For instance, for $N=2$, more realistic models are given by the restricted three-body problems, in which the centres are not fixed, but the two primaries move along a trajectory of the two-body problem (typically a circular or elliptic Keplerian orbit).  In this case, the only interaction that is neglected is the one between the third body and the cluster of the primaries; see for instance \cite{book_moser}. On the other hand, without any reduction, the study of the dynamics of multi-particle systems remains highly challenging, even after imposing symmetry constraints or reducing the dimensionality of the system \cite{FerTer2003,BaCaTe25,BaCaTe26}.
	
	It should be noted that also the Kepler problem, for similar reasons, gives a somewhat incomplete description of planetary motions and, for instance, does not account for precession phenomena, which can actually be observed. 
	
	A general-relativistic description provides the appropriate setting to account for these deviations.
	Once all symmetries and integrals of motion are taken in account, the problem can be reduced to a one-dimensional nonlinear equation for the radial motion, in complete analogy with the classical Kepler problem, the main difference being an effective potential with a power-law correction of order $r^{-3}$ (see for instance
	\cite[Chapter~12.3~B]{GoPoSa02}).  Other corrections to the effective potential have been considered in the literature, leading to a superposition of inverse-power laws with different coefficients; see \cite{BoDa26,NuCaLl91,LaLlNu91} for an account.
	
	A different approach, inspired by special relativity, has been proposed in \cite{AnBa71,Bo04,Ji13,LeMoPP,MuPa06}.
	Instead of tuning the effective potential, one replaces the classical
	momentum by the relativistic one, i.e., sets
	\begin{equation*}
		p = \frac{m \dot{q}}{\sqrt{1-\frac{\vert\dot{q}\vert^2}{c^2}}}
	\end{equation*}
	and considers the resulting dynamics.
	In other words, one defines a modified \textit{Lagrangian}
	\[
	L(q,\dot{q}) =- m c^2\sqrt{1-\frac{\vert\dot{q}\vert^2}{c^2}} +\frac{\mu}{\vert q\vert}
	\]
	and then looks at the associated Euler-Lagrange equations. There is  a natural conserved quantity along solutions, given by the total energy, that reads
	\[
	h_{\mathrm{rel}}=\frac{mc^2}{\sqrt{1-\frac{\vert \dot{q} \vert^2}{c^2}}}- \frac{\mu}{\vert q\vert} 
	\]
	as shown in Section \ref{sec:equivalence}.
	Despite possible theoretical limitations, this model provides a genuine flow on the whole space $\left(\mathbb{R}^2 \setminus\{0\}\right)\times \mathbb{R}^2$.	
	This approach has been extensively applied in \cite{BDM,BoDaFe24,BoFePa25}. A key observation of these works is that the \textit{relativistic dynamics} of the model is actually orbit-equivalent to a classical one, the only difference being the choice of potential. For instance, the new effective potential for the Kepler problem reads
	\[
	 V^{\mathrm{eff}}(q) = -
	 \frac{mc^2}{2}+
	 \frac{1}{2m c^2}
	 \left(
	 h_\mathrm{rel}+\frac{\mu}{|q|}
	 \right)^2.
	\]
	One can observe that the leading singular term is, in the terminology of singular
	Hamiltonian systems, a (critical) strong-force singularity. Thus, the relativistic correction transforms the classical Kepler singularity into
	a potential for which the variational framework based on the Maupertuis
	functional becomes especially effective.
		
	Relativistic Kepler-type problems have been investigated from several
	complementary perspectives. Periodic and quasi-periodic motions in central
	relativistic time-dependent force fields were studied in \cite{ToUrZa13}. More
	recently, variational methods have been used to obtain periodic solutions for
	relativistic Kepler equations under external perturbations
	\cite{BoFePa25}, while prescribed-energy periodic solutions
	bifurcating from invariant tori have been constructed in
	\cite{BoDaFe24}. The reduction above explains why singular
	variational techniques are naturally related to these relativistic models.
	
	The same observation about the correspondence between classical and relativistic systems applies to $N$-centre interactions, which are the superposition of Kepler potentials with different sources. In this case, the dynamics of a test particle is governed by a potential  of the form
	\[
	U(q) = \sum_{i=1}^{N} \frac{\mu_i}{\vert q-c_i \vert},
	\]
	where $c_i\in \mathbb{R}^2$ are the sources and $\mu_i \in \mathbb{R}$ the charges or the masses. 
	Using a relativistic correction as before, one obtains the following effective potential
	\[
	U_h^{\mathrm{eff}}(q) = \frac{1}{2 mc^2} \sum_{i=1}^N \frac{\mu_i^2}{\vert q -c_i\vert^2} +\frac{1}{2 m c^2} \sum_{i=1}^N \frac{\mu_i}{\vert q-c_i\vert}\left(2 h_\mathrm{rel}+\sum_{j\ne i} \frac{\mu_j}{\vert q-c_j \vert}\right),
	\]
	which again behaves as $\vert q-c_i\vert^{-2}$ near the singularities $c_i$ and as $\vert q\vert^{-1}$ at infinity.
	
	In this paper, we focus on a version of the  relativistic two-centre problem with potential given by
	\begin{equation}
		\label{eq:relativistic_potential}
		V_{\mathrm{rel}}(q)
		=
		\frac{\mu_1}{|q-c_1|}
		+\frac{\mu_2}{|q-c_2|}
		+
		\mathcal R(q),
	\end{equation}
	where \(\mathcal R\) denotes a perturbation which is smooth on $\mathbb{R}^{2}\setminus\{c_1,c_2\}$, bounded away from the centres, and such that $\mathcal{R}(q)\vert q-c_i\vert \to 0$ as $q\to c_i$.  Moreover, we assume that there exist constants $C,R>0$ and $\beta>1$ such that
	\begin{equation}\label{hyp:potential_rel}
		\max\left\lbrace\lvert\langle \nabla \mathcal{R}(q), Jq\rangle\lvert, |\nabla \mathcal{R}(q)|\right\rbrace \le \frac{C}{|q|^{\beta}},\quad \forall\, |q|\ge R,
	\end{equation}
	where $J$ is the matrix representing a rotation of angle $\pi/2$. 
	The equation of motion is then given by
	\begin{equation}
		\label{eq:motion_equation_relativistic}
	\frac{d}{dt}\left(\frac{m\dot{q}}{\sqrt{1-\frac{\vert \dot{q}\vert^2}{c^2}}}\right) = \nabla V_{\mathrm{rel}}(q).
	\end{equation}
	Hence, after fixing the energy $h_{\mathrm{rel}}$, the relativistic two-centre problem is
	equivalent to a generalised two-centre problem with strong-force
	singularities (see Section \ref{sec:equivalence} for details). This is the setting considered in the present paper.  In particular, we will study the case 
	\begin{equation}
		\label{eq:range_energy}
	h_{\mathrm{rel}}-mc^2 > - \inf_{q \in \mathbb{R}^2\setminus\{c_1,c_2\}} V_{\mathrm{rel}}(q).
	\end{equation}
	It is worth pointing out explicitly that the relativistic correction has a very strong effect on the dynamics and produces a dynamical system with completely different features. The classical two-centre problem is analytically integrable; the relativistic one is not. 
		
	The application of variational methods to singular systems, which is the approach we will follow in the present paper, has a long history  \cite{ACZ1990,AmbCotZel1993}.
	Strong-force singularities are particularly useful because they rule out
	collisions in minimisation procedures and allow one to construct
	collision-free trajectories in prescribed homotopy classes. In the planar
	two-centre case, Felmer and Tanaka gave a detailed description of periodic and scattering orbits and symbolic dynamics through minimisation of the
	Maupertuis functional \cite{FeTa00}.

	They consider, however, a
	singular potential behaving like $\vert q\vert^{-\alpha}$  with $\alpha$ strictly larger than \(2\), whereas the
	relativistic reduction naturally yields the critical exponent \(2\). 
	The chaotic behaviour of the systems with $\alpha =2$ and $N=2$ for generalised \(2\)-centre systems, was established through variational methods and
	Levi--Civita regularisation in \cite{BolKoz2017}, with the proof strategy consisting in building an invariant compact subset containing all periodic minimisers.
	However, although the strong-force
	condition with $\alpha=2$ yields collision-free minimisers, it does not imply that these minimizers localize, in the configuration space, away from the singularities.	
    It is thus natural to wonder whether it is possible to build compact invariant sets that are arbitrarily close to the singularities, which should persist if the strength $\alpha$ of the singularity is slightly weakened but are eventually bound to break down as $\alpha$ decreases. In this paper, we intend to answer this question in the affirmative and prove the existence of large families of scattering, trapping and periodic solutions. 

\subsection*{Main results}
In this section we outline the main contributions of the paper to the study of the relativistic  2-centre problem and describe the types of  scattering, trapping and periodic solutions we construct. However, before stating the main results, we need to introduce some terminology and a measure of the complexity of solutions, which is given by a suitable coding.

 More precisely, we associate to suitable trajectories a symbolic itinerary recording the successive transverse intersections with two distinguished oriented curves $\Gamma_1,\Gamma_2$. In this way, the dynamics is encoded by bi-infinite or one-sided sequences in the alphabet $\{\pm1,\pm2\}$, where the sign represents the orientation of each crossing (see Figure \ref{fig:itinery}).
 
 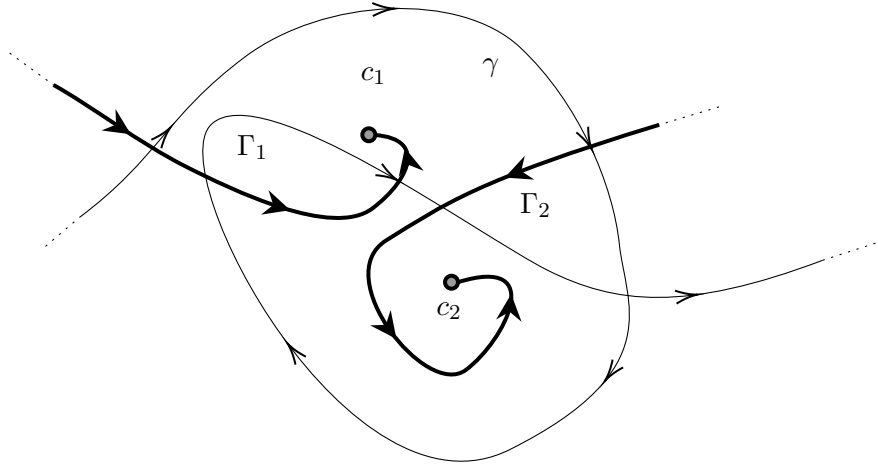
\begin{figure}[h]
 	\centering
 \begin{tikzpicture}[x=0.75pt,y=0.75pt,yscale=-1,xscale=1]
 	
 	\draw [line width=1.5]    (66,75.5) .. controls (85,86) and (108.5,104.38) .. (132.25,116.38) .. controls (156,128.38) and (206,152.5) .. (226,137.5) .. controls (246,122.5) and (255,99.5) .. (221,100.5) ;
 	\draw [shift={(104.83,100.27)}, rotate = 213.11] [fill={rgb, 255:red, 0; green, 0; blue, 0 }  ][line width=0.08]  [draw opacity=0] (13.4,-6.43) -- (0,0) -- (13.4,6.44) -- (8.9,0) -- cycle    ;
 	\draw [shift={(184.83,138.35)}, rotate = 196.98] [fill={rgb, 255:red, 0; green, 0; blue, 0 }  ][line width=0.08]  [draw opacity=0] (13.4,-6.43) -- (0,0) -- (13.4,6.44) -- (8.9,0) -- cycle    ;
 	\draw [shift={(242.38,107.08)}, rotate = 76.4] [fill={rgb, 255:red, 0; green, 0; blue, 0 }  ][line width=0.08]  [draw opacity=0] (13.4,-6.43) -- (0,0) -- (13.4,6.44) -- (8.9,0) -- cycle    ;
 	\draw  [fill={rgb, 255:red, 155; green, 155; blue, 155 }  ,fill opacity=1 ][line width=1.5]  (221,100.5) .. controls (221,98.71) and (222.46,97.25) .. (224.25,97.25) .. controls (226.04,97.25) and (227.5,98.71) .. (227.5,100.5) .. controls (227.5,102.29) and (226.04,103.75) .. (224.25,103.75) .. controls (222.46,103.75) and (221,102.29) .. (221,100.5) -- cycle ;
 	\draw  [fill={rgb, 255:red, 155; green, 155; blue, 155 }  ,fill opacity=1 ][line width=1.5]  (262.5,174.5) .. controls (262.5,172.71) and (263.96,171.25) .. (265.75,171.25) .. controls (267.54,171.25) and (269,172.71) .. (269,174.5) .. controls (269,176.29) and (267.54,177.75) .. (265.75,177.75) .. controls (263.96,177.75) and (262.5,176.29) .. (262.5,174.5) -- cycle ;
 	\draw  [dash pattern={on 0.84pt off 2.51pt}]  (44,59.5) -- (66,75.5) ;
 	\draw [line width=1.5]    (370,95.5) .. controls (278,123.5) and (262,135.5) .. (233,153.5) .. controls (204,171.5) and (253,233.5) .. (273,218.5) .. controls (293,203.5) and (314,160.5) .. (269,174.5) ;
 	\draw [shift={(292.46,121.98)}, rotate = 338.3] [fill={rgb, 255:red, 0; green, 0; blue, 0 }  ][line width=0.08]  [draw opacity=0] (13.4,-6.43) -- (0,0) -- (13.4,6.44) -- (8.9,0) -- cycle    ;
 	\draw [shift={(238.13,203)}, rotate = 233.03] [fill={rgb, 255:red, 0; green, 0; blue, 0 }  ][line width=0.08]  [draw opacity=0] (13.4,-6.43) -- (0,0) -- (13.4,6.44) -- (8.9,0) -- cycle    ;
 	\draw [shift={(295.71,180.45)}, rotate = 95.11] [fill={rgb, 255:red, 0; green, 0; blue, 0 }  ][line width=0.08]  [draw opacity=0] (13.4,-6.43) -- (0,0) -- (13.4,6.44) -- (8.9,0) -- cycle    ;
 	\draw  [dash pattern={on 0.84pt off 2.51pt}]  (400,86.5) -- (370,95.5) ;
 	\draw    (79,142) .. controls (119,112) and (122,91.5) .. (162,61.5) .. controls (202,31.5) and (270,28.5) .. (300,56.5) .. controls (330,84.5) and (347,126.5) .. (350,155.5) .. controls (353,184.5) and (375,218.5) .. (295,258.5) .. controls (215,298.5) and (120,120.5) .. (145,95.5) .. controls (170,70.5) and (280,151.5) .. (314,168.5) .. controls (348,185.5) and (383,189.5) .. (452,163.5) ;
 	\draw [shift={(124.71,96.82)}, rotate = 130.53] [color={rgb, 255:red, 0; green, 0; blue, 0 }  ][line width=0.75]    (10.93,-3.29) .. controls (6.95,-1.4) and (3.31,-0.3) .. (0,0) .. controls (3.31,0.3) and (6.95,1.4) .. (10.93,3.29)   ;
 	\draw [shift={(237.06,37.18)}, rotate = 177.46] [color={rgb, 255:red, 0; green, 0; blue, 0 }  ][line width=0.75]    (10.93,-3.29) .. controls (6.95,-1.4) and (3.31,-0.3) .. (0,0) .. controls (3.31,0.3) and (6.95,1.4) .. (10.93,3.29)   ;
 	\draw [shift={(336.11,107.74)}, rotate = 244.31] [color={rgb, 255:red, 0; green, 0; blue, 0 }  ][line width=0.75]    (10.93,-3.29) .. controls (6.95,-1.4) and (3.31,-0.3) .. (0,0) .. controls (3.31,0.3) and (6.95,1.4) .. (10.93,3.29)   ;
 	\draw [shift={(342.73,224.09)}, rotate = 306.5] [color={rgb, 255:red, 0; green, 0; blue, 0 }  ][line width=0.75]    (10.93,-3.29) .. controls (6.95,-1.4) and (3.31,-0.3) .. (0,0) .. controls (3.31,0.3) and (6.95,1.4) .. (10.93,3.29)   ;
 	\draw [shift={(183.42,205.19)}, rotate = 52.58] [color={rgb, 255:red, 0; green, 0; blue, 0 }  ][line width=0.75]    (10.93,-3.29) .. controls (6.95,-1.4) and (3.31,-0.3) .. (0,0) .. controls (3.31,0.3) and (6.95,1.4) .. (10.93,3.29)   ;
 	\draw [shift={(238.69,123.37)}, rotate = 209.88] [color={rgb, 255:red, 0; green, 0; blue, 0 }  ][line width=0.75]    (10.93,-3.29) .. controls (6.95,-1.4) and (3.31,-0.3) .. (0,0) .. controls (3.31,0.3) and (6.95,1.4) .. (10.93,3.29)   ;
 	\draw [shift={(389.16,180.71)}, rotate = 173.27] [color={rgb, 255:red, 0; green, 0; blue, 0 }  ][line width=0.75]    (10.93,-3.29) .. controls (6.95,-1.4) and (3.31,-0.3) .. (0,0) .. controls (3.31,0.3) and (6.95,1.4) .. (10.93,3.29)   ;
 	\draw  [dash pattern={on 0.84pt off 2.51pt}]  (62,156.5) -- (79,142) ;
 	\draw  [dash pattern={on 0.84pt off 2.51pt}]  (452,163.5) -- (486,152.5) ;
 	
 	\draw (157,99.4) node [anchor=north west][inner sep=0.75pt]    {$\Gamma _{1}$};
 	\draw (299,128.4) node [anchor=north west][inner sep=0.75pt]    {$\Gamma _{2}$};
 	\draw (219,65.4) node [anchor=north west][inner sep=0.75pt]    {$c_{1}$};
 	\draw (257,183.4) node [anchor=north west][inner sep=0.75pt]    {$c_{2}$};
 	\draw (280,61) node [anchor=north west][inner sep=0.75pt]   [align=left] {$\displaystyle \gamma $};
 	
 \end{tikzpicture}
 
 \label{fig:itinery}
 \caption{The (finite)  itinerary of the curve $\gamma$ is $(1,2,1,-1,2)$}
 \end{figure}
\begin{definition}[Itinerary of a curve]\label{def:itinerary}
	Let $\mathcal{D}$ be a compact subset of $\R^2$ and let $\Gamma_1,\Gamma_2\subset\mathcal{D}$ be two oriented non-intersecting curves. Let $\gamma\colon I\to\mathcal{D}$ be an oriented curve such that
	\begin{itemize}
		\item every intersection with $\Gamma_1\cup\Gamma_2$ is isolated and transverse;
		\item $\gamma$ intersects $\Gamma_1\cup\Gamma_2$ infinitely many times in both time directions.
	\end{itemize}
	Let $(t_i)_{i\in\Z}$ be the strictly increasing sequence of intersection times such that $\gamma(t_i)\in\Gamma_1\cup\Gamma_2$. For each $i$, we define the symbol $s_i\in\{\pm 1, \pm 2\}$ as follows:
	\begin{itemize}
		\item $s_i = l$ if $\gamma$ intersects $\Gamma_{|l|}$ at time $t_i$ with sign $\text{sgn}(l)$, where the sign of intersection is determined by the orientation of the pair $(\dot{\gamma}, \dot{\Gamma}_{|l|})$.
	\end{itemize}
	Then, the sequence $(s_i)_{i\in\Z}$ is called the \emph{bi-infinite itinerary} of the curve $\gamma$.
	
	Analogous definitions can be given for a \emph{positively infinite itinerary} $(s_i)_{i\in\N}$ or a \emph{finite itinerary} $(s_i)_{i=1}^k$.
\end{definition}

\begin{definition}[Admissible sequences]\label{def:admissible_k}
	Fix $k\in\N$. We say that a bi-infinite sequence $(s_i)_{i\in\Z}\in\{\pm1, \pm2\}^\Z$ is $k$-\emph{admissible} if 
	\begin{enumerate}
		\item there are no blocks consisting of the same symbol with length larger than $k$ in the sequence,
		\item there are no patterns of the form $(\pm 1, \mp 1)$ or $(\pm 2, \mp 2)$,
		\item there exist $i\neq j$ such that $|s_i|= 1$ and $|s_j|= 2$. 
	\end{enumerate}
	The same definition applies to a positively infinite sequence $(s_i)_{i\in\N}$ and a finite sequence $(s_i)_{i=1}^n$, with $n\ge 2$. 
\end{definition}

Recall that  the Lagrangian $L$ of the relativistic problem is given by
\[
L(q,\dot{q}) = - m c^2\sqrt{1-\frac{\vert \dot{q}\vert^2}{c^2}} +V_\mathrm{rel}(q)
\]
which is convex in the velocity. Thus, we can adopt a Hamiltonian point of view, as detailed in Section  \ref{sec:equivalence}.  This implies that the quantity
\begin{equation}
	\label{eq:energy_relativistic}
	h_{\mathrm{rel}}=\frac{mc^2}{\sqrt{1-\frac{\vert \dot{q} \vert^2}{c^2}}}-V_\mathrm{rel}(q) 
\end{equation}
is conserved along solutions of \eqref{eq:motion_equation_relativistic} and thus, we may restrict our attention to the energy hypersurface given by \eqref{eq:energy_relativistic}. 

Our main results show that the dynamics of \eqref{eq:motion_equation_relativistic} for $h_{\mathrm{rel}}$ satisfying \eqref{eq:range_energy} contains a rich symbolic subsystem: every admissible sequence is realised by a non-collisional solution of the equation of motion in \eqref{eq:motion_equation_relativistic}. Conversely, the qualitative behaviour of solutions is completely described by their itinerary. The first result establishes a full coding of bounded-type dynamics and yields chaotic subsystems in the unperturbed case.

\begin{theorem}[Coding]\label{thm:coding_rel}  
	Fix $h_{\mathrm{rel}}$ satisfying \eqref{eq:range_energy} and $k\in\N$. Let $(s_i)_{i\in\Z}\in\{\pm 1, \pm 2\}^\Z$ be a bi-infinite $k$-admissible sequence. 
	Then, there exists a solution $q\colon \R\to\R^2\setminus\{c_1,c_2\}$ of the motion and energy equations \eqref{eq:motion_equation_relativistic}-\eqref{eq:energy_relativistic} such that  $q$ has bi-infinite itinerary $(s_i)_{i\in\Z}$.
	
	Moreover, for fixed $k \in \mathbb{N}$, the collection of such orbits is contained in a compact subset of $\mathbb{R}^2\setminus\{c_1,c_2\}$. In particular, the initial conditions of these trajectories are contained in the set
		\[
		\Sigma_{\mathrm{rel}} = \left((\Gamma_1 \cup \Gamma_2)\times \mathbb{R}^2\right)\cap \{H_\mathrm{rel}=h_\mathrm{rel}\}.
		\]
	 The first return map $\Phi : \Sigma_{\mathrm{rel}}\to \Sigma_{\mathrm{rel}}$ is semi-conjugated with a sub-shift of finite type. Moreover, in the unperturbed case $\mathcal{R}=0$, the semi-conjugation is a conjugation.
\end{theorem}
Next, we turn to unbounded trajectories and scattering phenomena. In this case, solutions behave asymptotically like free particles at infinity, and one can prescribe their incoming and outgoing asymptotic directions, together with a finite symbolic coding in the interaction region.
\begin{theorem}[Scattering]\label{thm:scattering_bi_rel}
		Fix $h_{\mathrm{rel}}$ satisfying \eqref{eq:range_energy}, $n\in\N$ and $\theta^+, \theta^-\in[0,2\pi)$. Let $(s_i)_{i=1}^n\in\{\pm1, \pm2\}^n$ be a finite $(n+1)$-admissible sequence. 
	
	Then, there exists a solution $q\colon\R\to\R^2\setminus\{c_1,c_2\}$ of the motion and energy equations \eqref{eq:motion_equation_relativistic}-\eqref{eq:energy_relativistic} such that 
	\begin{itemize}
		\item $|q(t)|\to+\infty\quad\text{as}\ t\to\pm\infty$
		\item $\frac{q(t)}{|q(t)|}\to e^{i\theta^\pm}$ as $t\to\pm\infty$
	\end{itemize}
	and $q$ has finite itinerary $(s_i)_{i=1}^n$.
\end{theorem}
Finally, we construct entire solutions that are asymptotically free in one time direction and display a prescribed infinite symbolic behaviour in the other. These solutions provide a natural bridge between symbolic bounded-type dynamics and unbounded scattering trajectories.
\begin{theorem}[Trapping]\label{thm:trapped_relativistic}
		Fix $h_{\mathrm{rel}}$ satisfying \eqref{eq:range_energy}, $k\in\N$ and $\theta^-\in[0,2\pi)$. Let $(s_i)_{i\in\N}\in\{\pm1, \pm2\}^\N$ be a $k$-admissible sequence.
	
	Then, there exists a solution $q\colon\R\to\R^2\setminus\{c_1,c_2\}$ of the motion and energy equations \eqref{eq:motion_equation_relativistic}-\eqref{eq:energy_relativistic} such that
	\begin{itemize}
		\item $|q(t)|\to+\infty\quad\text{as}\ t\to-\infty$
		\item $\frac{q(t)}{|q(t)|}\to e^{i\theta^-}$ as $t\to-\infty$
	\end{itemize}
	and $q$ has positively infinite itinerary $(s_i)_{i\in\N}$.
\end{theorem}

\subsection*{Outline of the paper}

The paper is organised as follows. In Section \ref{sec:2_center_strong_force} we reformulate the relativistic problem as a classical two-center problem with strong-force.
In Section \ref{sec:topology} we introduce the topological framework and the combinatorial language used to encode homotopy classes of curves. In Section \ref{sec:variational} we set up the variational framework used throughout the paper. Sections \ref{sec:fixed_end} and \ref{sec:periodic} respectively study fixed-end minimisers and periodic minimisers of the Maupertuis functional, subject to the topological constraint imposed by a fixed admissible sequence. In Section \ref{sec:compactness} we provide a qualitative description of minimisers: in particular, we show that they remain at a bounded distance from the centres once the maximum number of turns around each centre is fixed. In Section \ref{sec:coding}, we introduce the symbolic framework and we prove Theorem \ref{thm:coding_rel} via an approximation argument that combines existence of periodic minimisers and compactness. Finally, Section \ref{sec:scattering} is devoted to scattering and trapped trajectories. We first analyse the asymptotic behaviour of solutions at infinity and prove Theorem \ref{thm:scattering_bi_rel}. These results are then combined with those in Section \ref{sec:coding} to prove Theorem \ref{thm:trapped_relativistic}.

\section{The high-energy two-center problem}
\label{sec:2_center_strong_force}
We have already pointed out the connection between the relativistic and classical problems, that will be discussed in detailed in Section \ref{sec:equivalence}. Here, we first focus on the classical two-center problem. 
More precisely, we fix two points $c_1, c_2\in\R^2$ and study a generalised and perturbed version of a 2-homogeneous two-centre problem. We thus introduce a potential function $V\in\mathcal{C}^2(\R^2\setminus\{c_1,c_2\})$, and we assume that there exists $r\in(0, |c_1-c_2|/2)$ such that
\begin{equation}\label{eq:potential}
V(x)=\frac{f_j(x)}{\lvert x-c_j\rvert^2}+ W_j(x),\quad\text{for any}\ x\in B_j=B_r(c_j),
\end{equation}
where $f_j$ is a positive, continuous function on $B_j$ that is smooth on $B_j\setminus\{c_j\}$, and $W_j$ is smooth and bounded on $B_j$. Additionally, we require that $V$ is bounded away from the centres, i.e., 
\[
\sup\limits_{x\in \R^2\setminus\{B_1\cup B_2\}} V(x)<+\infty.
\]
Moreover, to control its behaviour at infinity, we require that there exist $C, K>0$ and $\beta >1$ such that
\begin{equation}\label{hyp:infinity}
        \max\left\lbrace\lvert\langle \nabla V(x), Jx\rangle\lvert, |\nabla V(x)|\right\rbrace \le \frac{C}{|x|^{\beta}},\quad \forall\, |x|\ge K,
\end{equation}
where $J$ is the matrix representing a rotation rotation of angle $\pi/2$ . Note that assumption \eqref{hyp:infinity} guarantees that the dynamics is asymptotically free at infinity and allows one to define asymptotic directions for unbounded trajectories. 

Under these assumptions, we study the equation of motion
\begin{equation}\label{eq:motion}
\ddot{x}=\nabla V(x),
\end{equation}
whose classical solutions $x\colon I\to\R^2\setminus\{c_1,c_2\}$ satisfy the energy equation
\begin{equation}\label{eq:energy}
\frac12\lvert\dot{x}(t)\rvert^2-V(x(t))=h,\quad\text{for any}\ t\in I.
\end{equation}
For the purposes of this paper, we focus on the high-energy case in which 
\begin{equation}\label{hyp:energy}
    h > -\inf_{\R^2} V, 
\end{equation}
ensuring that $V+h >0$ throughout the configuration space. In this paper, we look for subsystems of \eqref{eq:motion} that display symbolic dynamics and we study scattering phenomena for unbounded solutions. 

We recall that fixed-energy classical solutions are also minimising geodesics of the \emph{Jacobi-Maupertuis metric}
\[
g_h(v,v) = (h+V(x))\lvert v\rvert^2,\quad v\in T_x\,\R^2.
\] 
For this reason, collision-free solutions of the two-centre problem can be found among critical points of the so-called \emph{Maupertuis functional}, which reads
\[
\mathcal{M}_h(x)=\frac12\int_0^1\lvert\dot{x}\rvert^2\int_0^1\left[V(x)+h\right].
\]
In particular, one could look for minimisers of $\mathcal{M}$ within homotopy classes of $\R^2\setminus\{c_1,c_2\}$ and use additional arguments to prove that they are collision-free, thus obtaining a critical point. 

In this setting we can formulate statements analogous to Theorems \ref{thm:coding_rel}, \ref{thm:scattering_bi_rel} and \ref{thm:trapped_relativistic}. We have thus the following results.

\begin{theorem}[Coding]\label{thm:coding}  
    Fix $h$ satisfying \eqref{hyp:energy} and $k\in\N$. Let $(s_i)_{i\in\Z}\in\{\pm 1, \pm 2\}^\Z$ be a bi-infinite $k$-admissible sequence. 
    Then, there exists a solution $x\colon \R\to\R^2\setminus\{c_1,c_2\}$ of the motion and energy equations \eqref{eq:motion}-\eqref{eq:energy} such that 
    $x$ has bi-infinite itinerary $(s_i)_{i\in\Z}$.

    Moreover, for fixed $k\in \mathbb{N}$, the collection of such orbits is contained in a compact subset of $\mathbb{R}^2\setminus\{c_1,c_2\}$. In particular  the initial conditions of these trajectories are contained in the set
    	\[
    	\Sigma = \left((\Gamma_1 \cup \Gamma_2)\times \mathbb{R}^2\right)\cap \{H=h\}.
    	\]
    	The first return map $\Phi : \Sigma\to \Sigma$ is semi-conjugated with a sub-shift of finite type. Moreover, in the unperturbed case $\mathcal{R}=0$, the semi-conjugation is a conjugation.
\end{theorem}

\begin{theorem}[Scattering]\label{thm:scattering_bi}
    Fix $h$ satisfying \eqref{hyp:energy}, $n\in\N$ and $\theta^+, \theta^-\in[0,2\pi)$. Let $(s_i)_{i=1}^n\in\{\pm1, \pm2\}^n$ be a finite $(n+1)$-admissible sequence. 

    Then, there exists a solution $x\colon\R\to\R^2\setminus\{c_1,c_2\}$ of the motion and energy equations \eqref{eq:motion}-\eqref{eq:energy} such that 
    \begin{itemize}
    \item $|x(t)|\to+\infty\quad\text{as}\ t\to\pm\infty$;
    \item $\frac{x(t)}{|x(t)|}\to e^{i\theta^\pm}$ as $t\to\pm\infty$.
    \end{itemize}
    and 
   Moreover,
    $x$ has finite itinerary $(s_i)_{i=1}^n$.
\end{theorem}

\begin{theorem}[Trapping]\label{thm:trapped}
    Fix $h$ satisfying \eqref{hyp:energy}, $k\in\N$ and $\theta^-\in[0,2\pi)$. Let $(s_i)_{i\in\N}\in\{\pm1, \pm2\}^\N$ be a $k$-admissible sequence.
    
    Then, there exists a solution $x\colon\R\to\R^2\setminus\{c_1,c_2\}$ of the motion and energy equations \eqref{eq:motion}-\eqref{eq:energy} such that
    \begin{itemize}
        \item $|x(t)|\to+\infty\quad\text{as}\ t\to-\infty$;
        \item $\frac{x(t)}{|x(t)|}\to e^{i\theta^-}$ as $t\to-\infty$.
    \end{itemize}
    Moreover,
    $x$ has positively infinite itinerary $(s_i)_{i\in\N}$.
\end{theorem}

\subsection{Equivalence between classical and relativistic version of the problem}
\label{sec:equivalence}

In this section we show that, up to a time-change, the two problems are indeed equivalent.
Recall that  the Lagrangian $L$ of the relativistic problem is given by
\[
L(q,v) = - m c^2\sqrt{1-\frac{\vert v\vert^2}{c^2}} +V_{\mathrm{rel}}(q), \quad V_{\mathrm{rel}}(q)=\frac{\mu_1}{|q-c_1|}+\frac{\mu_2}{|q-c_2|}+\mathcal R(q)
\]
which is convex in the velocity. Thus, we can adopt a Hamiltonian point of view, introduce the conjugate momentum  $	p = \partial_v L ={m v}/{\sqrt{1-\frac{\vert v\vert^2}{c^2}}}$ and write down the corresponding Hamiltonian $H_\mathrm{rel}$. This energy is then a conserved quantity along the motion  and, therefore, we work on an energy hyper-surface $\{H_{\mathrm{rel}} =h_{\mathrm{rel}}\}$ with $h_{\mathrm{rel}}$ satisfying \eqref{eq:range_energy}.  

First of all, let us observe that
\[
{\frac{\vert p\vert^2}{m^2 c^2}}  =  \frac{\vert v\vert^2/c^2}{(1-\frac{\vert v \vert^2}{c^2})} = \frac{\vert v\vert^2/c_2-1+1}{(1-\frac{\vert v \vert^2}{c^2})} = \frac{1}{1-\frac{\vert v \vert^2}{c^2}}-1
\]
and thus, after performing the Legendre transform of the Lagrangian, we obtain the following  expression for  the Hamiltonian $H_{\mathrm{rel}}$:
\begin{align*}
	H_{\mathrm{rel}}(p,q) &= \sup_{v:\vert v\vert <c}\left(\langle p,v\rangle - L(q,v)\right) = mc^2 \left(\frac{\vert v \vert ^2}{c^2\sqrt{1-\frac{\vert v\vert^2}{c^2}}}+\sqrt{1-\frac{\vert v \vert^2}{c^2}} \right)-V_{\mathrm{rel}}(q)\\
	& = \frac{mc^2}{\sqrt{1-\frac{\vert v \vert^2}{c^2}}}-V_{\mathrm{rel}}(q) = mc^2\sqrt{1+\frac{\vert p \vert^2}{m^2 c^2}}-V_{\mathrm{rel}}(q). 
\end{align*}
This implies, as observed in \cite{BDM,BoDa26}, that the quantity
\begin{equation*}
	h_{\mathrm{rel}}=\frac{mc^2}{\sqrt{1-\frac{\vert \dot{q} \vert^2}{c^2}}}-V_{\mathrm{rel}}(q) 
\end{equation*}
is conserved along solutions of \eqref{eq:motion_equation_relativistic}.
If we fix a level $h_{\mathrm{rel}}$ for the energy of the system, we have then that
\[
(h_{\mathrm{rel}}+V_{\mathrm{rel}}(q))^2 = (mc^2)^2\left(1+\frac{\vert p \vert^2}{m^2 c^2}\right) 
\]
and so the following relation holds
\[
\left(mc^2\sqrt{1+\frac{\vert p \vert^2}{m^2 c^2}}-V_{\mathrm{rel}}(q)-h_{\mathrm{rel}}\right)\left(mc^2\sqrt{1+\frac{\vert p \vert^2}{m^2c^2}}+V_{\mathrm{rel}}(q)+h_{\mathrm{rel}}\right)=
0.
\]
This means that, when $V_{\mathrm{rel}}(q)+h _{\mathrm{rel}}\ge 0$ the hypersurface $\{H_{\mathrm{rel}} = h_{\mathrm{rel}}\}$ can also be written as the zero-level set of the Hamiltonian
\[
K(p,q) =\frac{\vert p\vert^2}{2 m }-\frac{\left({V_{\mathrm{rel}}(q)+h_{\mathrm{rel}}}\right)^2}{2 mc^2}+\frac{m c^2}{2}
\]
This, equivalently, amounts to considering the energy level $-mc^2/2$ of the Hamiltonian $H$ defined as
\begin{equation}
	\label{eq:new_hamiltonian}
H(p,q) =   \frac{\vert p\vert^2}{2 m }-\frac{\left({V_{\mathrm{rel}}(q)+h_{\mathrm{rel}}}\right)^2}{2 mc^2}.
\end{equation}
Since Theorems \ref{thm:coding}-\ref{thm:trapped} apply when the Hill region has no boundary, we have to impose that 
\[
\frac{\left({V_{\mathrm{rel}}(q)+h_{\mathrm{rel}}}\right)^2}{2 mc^2}-\frac{mc^2}{2} =0, \text{ i.e., } \vert V_{\mathrm{rel}}(q)+h_{\mathrm{rel}}\vert = mc^2
\]
has no solution.
Since we are also assuming $h_{\mathrm{rel}}+V_{\mathrm{rel}}(q)> mc^2$ (compare with \eqref{eq:range_energy}), we conclude that the boundary of the Hill region for the two-center problem is empty. We now show that the conditions given in \eqref{eq:potential} are satisfied by the potential $V_\mathrm{rel}$ if the ones in \eqref{hyp:potential_rel} are.

\begin{lemma}
	Under conditions  \eqref{hyp:potential_rel} and \eqref{eq:energy_relativistic}, the Hamiltonian given in \eqref{eq:new_hamiltonian} corresponds to a perturbation of the two-center problem as in \eqref{eq:motion}. Moreover, assumptions  \eqref{hyp:infinity} and  \eqref{eq:energy}  are satisfied. 
	\begin{proof}
		If we consider $(V_\mathrm{rel}+h_\mathrm{rel})^2/2mc^2$, we obtain the following potential
		\[
		V(q) =  \frac{1}{2mc^2} \left( h_{\mathrm{rel}} + \frac{ \mu_1}{\vert q-c_1\vert}+ \frac{\mu_2}{\vert q-c_2\vert}+\mathcal{R}(q)\right)^2.
		\]
		
		We have already verified that the energy condition \eqref{hyp:energy} holds provided that \eqref{eq:range_energy} does. 
		
		Clearly the potential $V$ has the form \eqref{eq:potential} for a suitable choice of functions $f_j $  and $W_j$ and the potential, far off from the centers is bounded as well. We just need to verify the assumption \eqref{hyp:infinity}. Let us fix  a ball of radius $R>0$ containing the centers $c_1$ and $c_2$. There exist constants $C_1,C_2>0$ such that, for $R$ large enough
		\[
		\vert \nabla V_\mathrm{rel}\vert \le C_1\frac{\mu_1+\mu_2}{\vert q\vert^2}+\vert \nabla \mathcal{R}\vert \le C_2 \frac{1}{\vert q\vert^{\min\{\beta,2\}}},
		\]
		since the perturbation $\mathcal{R}$ satisfies 
		\(\vert \nabla \mathcal{R}(q)\vert\le C \vert q\vert^{-\beta}
		\)
		with $\beta>1$. This, together with condition \eqref{eq:range_energy}, implies that the gradient of $V$ satisfies
		\[
			\vert \nabla V \vert  = \frac{1}{mc^2} (V_\mathrm{rel}+h_\mathrm{rel})\vert \nabla V_\mathrm{rel}\vert \le C_3 \frac{1}{\vert q\vert^{\min\{\beta,2\}}}.
		\]
		Let us show now that, $\langle \nabla V(q),J q \rangle$ has the same rate of decay. We have that
		\[
		\langle J q,\nabla \left(\frac{\mu_1}{\vert q-c_i\vert}\right)\rangle = \langle Jq, \frac{q-c_i}{\vert q-c_i\vert^3}\rangle = -\frac{\langle c_i,q\rangle}{\vert q-c_i\vert^3}
		\]
		and so there exists a constant $C_4$ such that
		\[
		\Big\vert \langle J q,\nabla \left(\frac{\mu_1}{\vert q-c_i\vert}\right)\rangle \Big\vert  \le \frac{
		C_4}{\vert q\vert^2} 
		\]
		Thus, thanks to \eqref{hyp:potential_rel}, we can conclude that there exists a constant $C_5>0$ such that
		\[
		\vert\langle \nabla V(q),J q \rangle\vert  =2(V_\mathrm{rel}(q)+h_\mathrm{rel})\vert \langle \nabla V_\mathrm{rel}(q),Jq\rangle \vert \le C_5 \vert q\vert^{-\beta'} 
		\]		
		for some $\beta'>1$ and thus \eqref{eq:potential} is satisfied.
	\end{proof}
\end{lemma}

\section{Topological and variational framework}

\subsection{Homotopy classes of curves in the punctured plane}\label{sec:topology}

Here, we briefly recall the topological structure of the punctured plane and introduce the combinatorial language used throughout the paper to encode homotopy classes of curves. The configuration space $\R^2\setminus\{c_1,c_2\}$ is path-connected and homotopically equivalent to a wedge of two circles. Therefore, its fundamental group is the free group with two generators corresponding to counterclockwise loops around $c_1$ and $c_2$, denoted by $\alpha_1$ and $\alpha_2$.

For the purposes of this paper, we will employ generators in a convenient way to encode free homotopy classes of loops. It is  well known that homotopy classes correspond to conjugacy classes in $\pi_1(\mathbb{R}\setminus\{c_1,c_2\})$ or, equivalently, that they can be represented by a reduced word in the alphabet (see Figure \ref{fig:generators})
\begin{figure}[t]
    \centering
	\resizebox{.4\textwidth}{!}{\tikzset{every picture/.style={line width=0.75pt}} 

\begin{tikzpicture}[x=0.75pt,y=0.75pt,yscale=-1,xscale=1]

\draw  [color={rgb, 255:red, 2; green, 54; blue, 118 }  ,draw opacity=1 ] (96,130.5) .. controls (96,100.4) and (120.4,76) .. (150.5,76) .. controls (180.6,76) and (205,100.4) .. (205,130.5) .. controls (205,160.6) and (180.6,185) .. (150.5,185) .. controls (120.4,185) and (96,160.6) .. (96,130.5) -- cycle ;
\draw  [color={rgb, 255:red, 2; green, 54; blue, 118 }  ,draw opacity=1 ] (206,130.5) .. controls (206,100.4) and (230.4,76) .. (260.5,76) .. controls (290.6,76) and (315,100.4) .. (315,130.5) .. controls (315,160.6) and (290.6,185) .. (260.5,185) .. controls (230.4,185) and (206,160.6) .. (206,130.5) -- cycle ;
\draw  [fill={rgb, 255:red, 128; green, 128; blue, 128 }  ,fill opacity=1 ][line width=2.25]  (146.86,131.18) .. controls (146.46,129.18) and (147.76,127.25) .. (149.77,126.87) .. controls (151.78,126.5) and (153.73,127.81) .. (154.14,129.82) .. controls (154.54,131.82) and (153.24,133.75) .. (151.23,134.13) .. controls (149.22,134.5) and (147.27,133.19) .. (146.86,131.18) -- cycle ;
\draw  [fill={rgb, 255:red, 128; green, 128; blue, 128 }  ,fill opacity=1 ][line width=2.25]  (256.86,131.18) .. controls (256.46,129.18) and (257.76,127.25) .. (259.77,126.87) .. controls (261.78,126.5) and (263.73,127.81) .. (264.14,129.82) .. controls (264.54,131.82) and (263.24,133.75) .. (261.23,134.13) .. controls (259.22,134.5) and (257.27,133.19) .. (256.86,131.18) -- cycle ;
\draw  [color={rgb, 255:red, 2; green, 54; blue, 118 }  ,draw opacity=1 ] (118.23,93.09) -- (106.9,98.24) -- (109.58,86.09) ;
\draw  [color={rgb, 255:red, 2; green, 54; blue, 118 }  ,draw opacity=1 ] (239.11,86.5) -- (226.75,87.97) -- (232.99,77.2) ;
\draw  [fill={rgb, 255:red, 0; green, 0; blue, 0 }  ,fill opacity=1 ][line width=2.25]  (203.74,130.93) .. controls (203.48,129.68) and (204.29,128.48) .. (205.54,128.24) .. controls (206.79,128.01) and (208.01,128.83) .. (208.26,130.07) .. controls (208.52,131.32) and (207.71,132.52) .. (206.46,132.76) .. controls (205.21,132.99) and (203.99,132.17) .. (203.74,130.93) -- cycle ;

\draw (152.5,133.5) node [anchor=north west][inner sep=0.75pt]   [align=left] {$\displaystyle c_{1}$};
\draw (262.5,133.5) node [anchor=north west][inner sep=0.75pt]   [align=left] {$\displaystyle c_{2}$};
\draw (96,173) node [anchor=north west][inner sep=0.75pt]  [color={rgb, 255:red, 2; green, 54; blue, 118 }  ,opacity=1 ] [align=left] {$\displaystyle \alpha _{1}$};
\draw (296,174) node [anchor=north west][inner sep=0.75pt]  [color={rgb, 255:red, 2; green, 54; blue, 118 }  ,opacity=1 ] [align=left] {$\displaystyle \alpha _{2}$};

\end{tikzpicture}}
	\caption{The generators of the fundamental group $\pi_1(\R^2\setminus\{c_1,c_2\})$.}\label{fig:generators}
\end{figure}
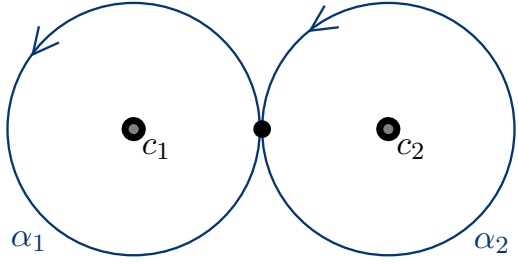
\begin{equation}\label{rem:free_group}
    \{\al_1^{\pm1}, \al_2^{\pm1}\}.
\end{equation}
Here reduced means that no pattern of the form $\alpha_1^{\pm} \alpha_1^{\mp}$ or $\alpha_2^\pm \alpha_2^{\mp}$ can appear and finite words are identified up to cyclic shift.

For a given loop $\gamma$, we denote its homotopy class as $[\gamma]$. A loop is said to be \emph{contractible} if it is homotopic to a point. Finally, we introduce the operation of \emph{concatenation} bewteen any two paths in $\R^2\setminus\{c_1,c_2\}$. 
\begin{definition}
    For two paths $\gamma\colon[t_0,t_1]\to\R^2\setminus\{c_1,c_2\}$ and $\phi\colon[\tau_0,\tau_1]\to\R^2\setminus\{c_1,c_2\}$, if $\gamma(t_1)=\phi(\tau_0)$ we can define their \emph{concatenation} as the path $\gamma\#\phi\colon[t_0,\tau_1]\to\R^2\setminus\{c_1,c_2\}$. 
\end{definition}
For the purposes of this paper, we need to define the geometric self-intersections number of paths. Note that we will always refer to \emph{unsigned} intersections. Our main reference is \cite{HasSco1985}.
\begin{definition}[Intersection number]
    For two loops $\gamma,\phi$ on $\R^2\setminus\{c_1,c_2\}$, we define their \emph{intersection number} as
    \[
    \lvert\gamma\cap\phi\rvert = \lvert \left\lbrace(t_1,t_2)\in\cerchio\times\cerchio:\,\gamma(t_1)=\phi(t_2)\right\rbrace\rvert\in\N\cup\{+\infty\}.
    \]
    For a loop $\gamma$ on $\R^2\setminus\{c_1,c_2\}$, we define its \emph{self-intersection number} as 
    \[
    \lvert\gamma\rvert = \frac12\lvert\left\lbrace (t_1,t_2)\in\cerchio\times\cerchio:\,t_1\neq t_2,\ \gamma(t_1)=\gamma(t_2)\right\rbrace\rvert\in\N\cup\{+\infty\}.
    \]
\end{definition}

Note that the two numbers in the previous definition are always finite when we consider \emph{general position} curves, i.e., curves which have only transversal intersections.

\begin{definition}[Taut loops]\label{def:taut}
    Given two loops $\gamma,\phi$ in $\R^2\setminus\{c_1,c_2\}$ we say that they are in \emph{minimal position} if they realise the minimum number of intersections in their homotopy classes $[\gamma], [\phi]$. 
    
    Similarly, a loop $\gamma$ is said to be in \emph{minimal position} or \emph{taut} if it realises the minimum number of self-intersections in its homotopy class $[\gamma]$. 
\end{definition}

Minimal position representatives can be assumed to have only transverse intersections, since tangential intersections can be removed by homotopies which decrease the number of intersections. However, a loop in general position does not generally realise the minimum number of self-intersections. This can be achieved by performing a finite sequence of \emph{moves} that remove the redundant self-intersections corresponding to \emph{singular 1-gons} or \emph{2-gons} (see Figure \ref{fig:gons}).
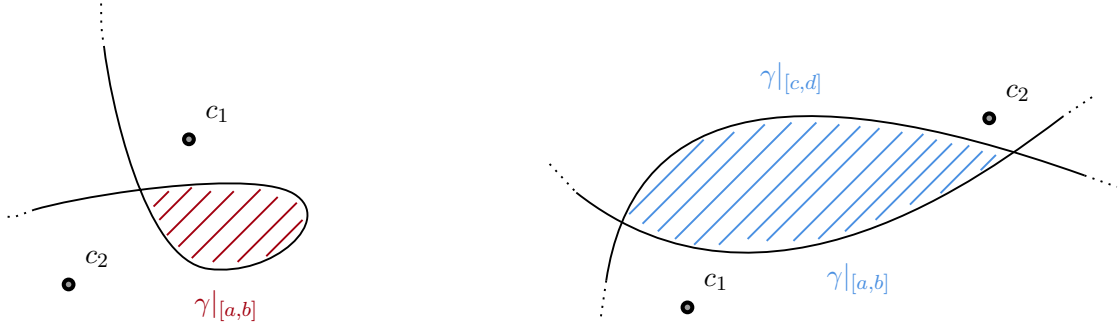
\begin{figure}[t]
    \centering
	\resizebox{.85\textwidth}{!}{\tikzset{every picture/.style={line width=0.75pt}} 

\begin{tikzpicture}[x=0.75pt,y=0.75pt,yscale=-1,xscale=1]

\draw    (81.16,36.04) .. controls (84.12,61.06) and (101.45,147.24) .. (135.69,151.63) .. controls (169.93,156.03) and (202.22,125.79) .. (180.55,112.43) .. controls (167.67,104.49) and (131.15,106.19) .. (98.22,110.49) .. controls (75.72,113.43) and (54.9,117.58) .. (44.44,120.7) ;
\draw  [dash pattern={on 0.84pt off 2.51pt}]  (81.16,36.04) .. controls (79.17,21.58) and (79.8,31.83) .. (79.53,12.3) ;
\draw  [dash pattern={on 0.84pt off 2.51pt}]  (44.44,120.7) .. controls (34.24,124.26) and (35.71,124.73) .. (30.33,124.32) ;
\draw  [fill={rgb, 255:red, 0; green, 0; blue, 0 }  ,fill opacity=1 ] (122.41,83.8) .. controls (122.41,82.09) and (123.8,80.7) .. (125.52,80.7) .. controls (127.23,80.7) and (128.62,82.09) .. (128.62,83.8) .. controls (128.62,85.52) and (127.23,86.91) .. (125.52,86.91) .. controls (123.8,86.91) and (122.41,85.52) .. (122.41,83.8) -- cycle ;
\draw  [fill={rgb, 255:red, 155; green, 155; blue, 155 }  ,fill opacity=1 ] (123.56,83.8) .. controls (123.56,82.72) and (124.44,81.85) .. (125.52,81.85) .. controls (126.6,81.85) and (127.47,82.72) .. (127.47,83.8) .. controls (127.47,84.89) and (126.6,85.76) .. (125.52,85.76) .. controls (124.44,85.76) and (123.56,84.89) .. (123.56,83.8) -- cycle ;
\draw    (344,157.71) .. controls (357.33,70.71) and (431.33,44.71) .. (594.33,102.71) ;
\draw    (328.33,111.71) .. controls (427.33,189.71) and (543.33,101.71) .. (583.33,71.71) ;
\draw  [dash pattern={on 0.84pt off 2.51pt}]  (314.33,96.71) -- (328.33,111.71) ;
\draw  [dash pattern={on 0.84pt off 2.51pt}]  (598.33,61.71) -- (583.33,71.71) ;
\draw  [dash pattern={on 0.84pt off 2.51pt}]  (594.33,102.71) -- (611.33,108.71) ;
\draw  [dash pattern={on 0.84pt off 2.51pt}]  (344,157.71) -- (341,178) ;
\draw [color={rgb, 255:red, 163; green, 0; blue, 16 }  ,draw opacity=1 ]   (110,126) -- (127,109) ;
\draw [color={rgb, 255:red, 163; green, 0; blue, 16 }  ,draw opacity=1 ]   (114,134) -- (137,111) ;
\draw [color={rgb, 255:red, 163; green, 0; blue, 16 }  ,draw opacity=1 ]   (120,140) -- (149,111) ;
\draw [color={rgb, 255:red, 163; green, 0; blue, 16 }  ,draw opacity=1 ]   (127,146) -- (163,110) ;
\draw [color={rgb, 255:red, 163; green, 0; blue, 16 }  ,draw opacity=1 ]   (138,148) -- (172,114) ;
\draw [color={rgb, 255:red, 163; green, 0; blue, 16 }  ,draw opacity=1 ]   (151,148) -- (182,117) ;
\draw [color={rgb, 255:red, 163; green, 0; blue, 16 }  ,draw opacity=1 ]   (166,145) -- (185,126) ;
\draw [color={rgb, 255:red, 163; green, 0; blue, 16 }  ,draw opacity=1 ]   (107,119) -- (115,111) ;
\draw [color={rgb, 255:red, 74; green, 144; blue, 226 }  ,draw opacity=1 ]   (356,123) -- (395,84) ;
\draw [color={rgb, 255:red, 74; green, 144; blue, 226 }  ,draw opacity=1 ]   (363,128) -- (411,80) ;
\draw [color={rgb, 255:red, 74; green, 144; blue, 226 }  ,draw opacity=1 ]   (371,132) -- (427,76) ;
\draw [color={rgb, 255:red, 74; green, 144; blue, 226 }  ,draw opacity=1 ]   (380,136) -- (441,75) ;
\draw [color={rgb, 255:red, 74; green, 144; blue, 226 }  ,draw opacity=1 ]   (390,139) -- (454,75) ;
\draw [color={rgb, 255:red, 74; green, 144; blue, 226 }  ,draw opacity=1 ]   (403,139) -- (466,76) ;
\draw [color={rgb, 255:red, 74; green, 144; blue, 226 }  ,draw opacity=1 ]   (418,136) -- (477,77) ;
\draw [color={rgb, 255:red, 74; green, 144; blue, 226 }  ,draw opacity=1 ]   (428,137) -- (487,78) ;
\draw [color={rgb, 255:red, 74; green, 144; blue, 226 }  ,draw opacity=1 ]   (442,137) -- (499,80) ;
\draw [color={rgb, 255:red, 74; green, 144; blue, 226 }  ,draw opacity=1 ]   (455,136) -- (510,81) ;
\draw [color={rgb, 255:red, 74; green, 144; blue, 226 }  ,draw opacity=1 ]   (469,132) -- (517,84) ;
\draw [color={rgb, 255:red, 74; green, 144; blue, 226 }  ,draw opacity=1 ]   (485,126) -- (525,86) ;
\draw [color={rgb, 255:red, 74; green, 144; blue, 226 }  ,draw opacity=1 ]   (503,118) -- (534,87) ;
\draw [color={rgb, 255:red, 74; green, 144; blue, 226 }  ,draw opacity=1 ]   (519,112) -- (541,90) ;
\draw [color={rgb, 255:red, 74; green, 144; blue, 226 }  ,draw opacity=1 ]   (540,100) -- (548,92) ;
\draw  [fill={rgb, 255:red, 0; green, 0; blue, 0 }  ,fill opacity=1 ] (59.41,159.8) .. controls (59.41,158.09) and (60.8,156.7) .. (62.52,156.7) .. controls (64.23,156.7) and (65.62,158.09) .. (65.62,159.8) .. controls (65.62,161.52) and (64.23,162.91) .. (62.52,162.91) .. controls (60.8,162.91) and (59.41,161.52) .. (59.41,159.8) -- cycle ;
\draw  [fill={rgb, 255:red, 155; green, 155; blue, 155 }  ,fill opacity=1 ] (60.56,159.8) .. controls (60.56,158.72) and (61.44,157.85) .. (62.52,157.85) .. controls (63.6,157.85) and (64.47,158.72) .. (64.47,159.8) .. controls (64.47,160.89) and (63.6,161.76) .. (62.52,161.76) .. controls (61.44,161.76) and (60.56,160.89) .. (60.56,159.8) -- cycle ;
\draw  [fill={rgb, 255:red, 0; green, 0; blue, 0 }  ,fill opacity=1 ] (541.41,72.8) .. controls (541.41,71.09) and (542.8,69.7) .. (544.52,69.7) .. controls (546.23,69.7) and (547.62,71.09) .. (547.62,72.8) .. controls (547.62,74.52) and (546.23,75.91) .. (544.52,75.91) .. controls (542.8,75.91) and (541.41,74.52) .. (541.41,72.8) -- cycle ;
\draw  [fill={rgb, 255:red, 155; green, 155; blue, 155 }  ,fill opacity=1 ] (542.56,72.8) .. controls (542.56,71.72) and (543.44,70.85) .. (544.52,70.85) .. controls (545.6,70.85) and (546.47,71.72) .. (546.47,72.8) .. controls (546.47,73.89) and (545.6,74.76) .. (544.52,74.76) .. controls (543.44,74.76) and (542.56,73.89) .. (542.56,72.8) -- cycle ;
\draw  [fill={rgb, 255:red, 0; green, 0; blue, 0 }  ,fill opacity=1 ] (383.41,171.8) .. controls (383.41,170.09) and (384.8,168.7) .. (386.52,168.7) .. controls (388.23,168.7) and (389.62,170.09) .. (389.62,171.8) .. controls (389.62,173.52) and (388.23,174.91) .. (386.52,174.91) .. controls (384.8,174.91) and (383.41,173.52) .. (383.41,171.8) -- cycle ;
\draw  [fill={rgb, 255:red, 155; green, 155; blue, 155 }  ,fill opacity=1 ] (384.56,171.8) .. controls (384.56,170.72) and (385.44,169.85) .. (386.52,169.85) .. controls (387.6,169.85) and (388.47,170.72) .. (388.47,171.8) .. controls (388.47,172.89) and (387.6,173.76) .. (386.52,173.76) .. controls (385.44,173.76) and (384.56,172.89) .. (384.56,171.8) -- cycle ;

\draw (132.67,64.79) node [anchor=north west][inner sep=0.75pt]   [align=left] {$\displaystyle c_{1}$};
\draw (69.67,140.79) node [anchor=north west][inner sep=0.75pt]   [align=left] {$\displaystyle c_{2}$};
\draw (551.67,53.79) node [anchor=north west][inner sep=0.75pt]   [align=left] {$\displaystyle c_{2}$};
\draw (393.67,152.79) node [anchor=north west][inner sep=0.75pt]   [align=left] {$\displaystyle c_{1}$};
\draw (127,161) node [anchor=north west][inner sep=0.75pt]  [color={rgb, 255:red, 163; green, 0; blue, 16 }  ,opacity=1 ] [align=left] {$\displaystyle \gamma |_{[ a,b]}$};
\draw (458,147) node [anchor=north west][inner sep=0.75pt]  [color={rgb, 255:red, 74; green, 144; blue, 226 }  ,opacity=1 ] [align=left] {$\displaystyle \gamma |_{[ a,b]}$};
\draw (423,42) node [anchor=north west][inner sep=0.75pt]  [color={rgb, 255:red, 74; green, 144; blue, 226 }  ,opacity=1 ] [align=left] {$\displaystyle \gamma |_{[ c,d]}$};

\end{tikzpicture}}
	\caption{Singular 1-gon and 2-gon, cf. Definition \ref{def:gons}.}\label{fig:gons}
\end{figure}
\begin{definition}\label{def:gons}
    Let $\gamma\colon\cerchio\to\R^2\setminus\{c_1,c_2\}$ be a loop. We say that:
    \begin{itemize}
        \item $\gamma$ has a \emph{singular 1-gon} if there exists a sub-arc $[a,b]\subset\cerchio$ such that $\gamma(a)=\gamma(b)$ and $\gamma|_{[a,b]}$ is contractible.
        \item $\gamma$ has a \emph{singular 2-gon} if there exist two disjoint sub-arcs $[a,b],[c,d]\subset\cerchio$ such that $\gamma(a)=\gamma(c)$, $\gamma(b)=\gamma(d)$ and the loop obtained by concatenating $\gamma|_{[a,b]}$ and $\gamma|_{[c,d]}$ is contractible.
    \end{itemize}    
\end{definition}

 To determine the minimum number of self-intersections of a loop $\gamma$ in its homotopy class: it is enough to take a representative in $[\gamma]$ which is in general position and to repeatedly \emph{remove} singular 1-gons and 2-gons. This is the content of the following classical result
\begin{theorem}[\cite{HasSco1985}, Theorem 4.2]\label{thm:taut}
    Let $\gamma\colon\cerchio\to\R^2\setminus\{c_1,c_2\}$ be a loop in general position. If $\gamma$ is not taut, then it posseses either a singular 1-gon or a singular 2-gon.
\end{theorem}

We conclude this section with two further definitions useful in the context of this paper.
\begin{definition}[Simple loops]
    We say that a loop $\gamma$ in $\R^2\setminus\{c_1,c_2\}$ is \emph{simple} if it is free of self-intersections. A sub-loop of $\gamma$ is said to be \emph{innermost} if it is simple. 
    
    A singular 1-gon or 2-gon is \emph{innermost} if, regarded as a loop, it does not contain any singular 1-gons or 2-gons. 
\end{definition}

\subsection{Variational framework}\label{sec:variational}

In this paper, we study both periodic and fixed-end solutions of the motion equation \eqref{eq:motion} at fixed energy levels. Their existence will be established via variational methods, so in this section we introduce the functional setting and the variational principle used in the proof.

Let $J=[t_0,t_1]\subset\R$ and consider the Sobolev space $H^1(J;\R^2)$. We denote by $H^1(J;\R^2\setminus\{c_1,c_2\})$ the open subset of $H^1(J;\R^2)$ consisting of paths avoiding the centres, i.e., paths whose image does not intersect $\{c_1,c_2\}$.

For $h$ satisfying \eqref{hyp:energy}, we introduce the \emph{Maupertuis functional}
\[
    \begin{aligned}
    &\mathcal{M}_h\colon H^1(J;\R^2)\to\R\cup\{+\infty\} \\
    &\mathcal{M}_h(\gamma)=\frac12\int_J\lvert\dot{\gamma}(t)\rvert^2\,dt\int_J\left[h+V(\gamma(t))\right]\,dt,
    \end{aligned}
\]
and, for any $\gamma\in H^1(J;\R^2)$ such that $\mathcal{M}_h(\gamma)>0$, we set
\[
\omega^2=\dfrac{\int_J\left[h+V(\gamma(t))\right]\,dt}{\frac12\int_J\lvert\dot{\gamma}(t)\rvert^2\,dt}\quad \text{and}\quad J_\omega=\left[\frac{t_0}{\omega},\frac{t_1}{\omega}\right].
\]
It is well-known that $\mathcal{M}_h$ is differentiable on $H^1(J;\R^2\setminus\{c_1,c_2\})$ (see, e.g., \cite{AmbCotZel1993}) and its critical points are solutions of equation \eqref{eq:motion} at energy $h$:
\begin{proposition}[The Maupertuis principle]\label{prop:maupertuis}
    Let $\gamma\in H^1(J;\R^2\setminus\{c_1,c_2\})$ be a non constant critical point of $\mathcal{M}_h$ such that $\mathcal{M}_h(\gamma)>0$. Then, $\gamma$ is a $C^2$ curve that solves the problem
    \[
    \begin{cases}
        \begin{aligned}
        &\omega^2\ddot{\gamma}(t)=\nabla V(\gamma(t)) & t\in J \\
        &\frac12\lvert\dot{\gamma}(t)\rvert^2 -\frac{V(\gamma(t))}{\omega^2}=\frac{h}{\omega^2} & t\in J
        \end{aligned}
    \end{cases}
    \]
    with boundary conditions $(\gamma(t_0),\gamma(t_1))$. Similarly, the function $\psi(t)=\gamma(\omega t)$ solves the problem 
    \[
    \begin{cases}
        \begin{aligned}
            &\ddot{\psi}(t)=\nabla V(\psi(t)) & t\in J_\omega \\
            &\frac12\lvert\dot{\psi}(t)\rvert^2 -V(\psi(t)) = h & t\in J_\omega
        \end{aligned}
    \end{cases}
    \]
    with boundary conditions $(\psi(t_0/\omega), \psi(t_1/\omega))$. 
\end{proposition}

The previous result can be adapted to the search for both periodic and fixed-end solutions, by imposing suitable boundary conditions and, if necessary, additional topological constraint within $ H^1(J;\R^2\setminus\{c_1,c_2\})$ (see Sections \ref{sec:fixed_end}-\ref{sec:periodic}). In any case, the space $ H^1(J;\R^2\setminus\{c_1,c_2\})$ is in general not closed and direct methods cannot be applied directly in this setting. For this reason, we shall always begin by working in the ambient Hilbert space $H^1(J;\R^2)$, which also includes paths colliding with the centres, and then deduce the smoothness of minimisers through qualitative arguments. 

In the course of this analysis, we will need to employ some properties of the Maupertuis functional, which are collected below. For the proof of these facts, we refer to \cite[Lemma 3.5]{BarCan}. It is worth noting that the Maupertuis functional is \emph{not} additive under concatenation; given $\gamma_1,\gamma_2\in H^1(J;\R^2)$ for which the concatenation $\gamma_1\#\gamma_2$ is well-defined, in general
\[
\mathcal{M}_h(\gamma_1\#\gamma_2)\neq\mathcal{M}_h(\gamma_1)+\mathcal{M}_h(\gamma_2).
\]
Another possible functional to obtain solutions of \eqref{eq:motion}, this time \textit{unparametrized}, is the \emph{Jacobi-length} one. This functional, albeit non-differentiable, is additive and it is  defined as
\[
    \mathcal{L}_h(\gamma) = \int_J\lvert\dot{\gamma}(t)\rvert\sqrt{h+V(\gamma(t))}\,dt.
\]
Observe that $\mathcal{L}_h$ is well-defined for any $\gamma\in H^1(J;\R^2\setminus\{c_1,c_2\})$, since $V(\gamma)+h>0$ thanks to \eqref{hyp:energy}. 
\begin{proposition}\label{prop:maupertuis_properties}
    The Maupertuis functional satisfies the following properties: 
    \begin{enumerate}[label = \roman*)]
        \item $\mathcal{M}_h$ is invariant with respect to affine time reparametrisations;
        \item $\mathcal{M}_h$ is super-additive: if $\gamma\in H^1(J;\R^2)$ and $[a,b]\subset J$ then 
        \[
        \mathcal{M}_h(\gamma)\ge \mathcal{M}_h(\gamma|_{[a,b]})+\mathcal{M}_h(\gamma|_{J\setminus[a,b]});
        \]
        \item up to reparametrisations, $\gamma$ is a critical point of $\mathcal{M}_h$ in $ H^1(J;\R^2\setminus\{c_1,c_2\})$ if and only if it is a critical point of $\mathcal{L}_h$ in the same space.
    \end{enumerate}
\end{proposition}

\section{Existence of collision-less minimisers}

\subsection{Fixed-end minimisation}\label{sec:fixed_end}

In this section we want to study the family of fixed-end problems
\[
    \begin{cases}
        \ddot{x} = \nabla V(x),\\
        \frac12 \lvert\dot{x}\rvert^2 - V(x) = h,\\
        x(\pm \omega) = q^\pm
    \end{cases}
\]
with $\omega>0$, $h$ satisfying \eqref{hyp:energy} and $q^\pm\in\R^2\setminus\{c_1,c_2\}$. For the purposes of this paper and to ease its exposition, we will assume that 
\begin{equation}\label{def:fixed_end}
    \lvert q^+\rvert = \lvert q^-\rvert\quad\text{and}\quad |q^+|>|c_i|,\ \text{for}\ i=1,2,
\end{equation}
so that the two centres lie inside the ball of radius $|q^+|$ centred at $0$. We can introduce the Maupertuis functional 
\[
    \mathcal{M}_h(\gamma) = \frac12\int_{-1}^{1}\lvert\dot{\gamma}(t)\rvert^2\,dt\int_{-1}^{1}\left[h+ V(\gamma(t))\right]\,dt,
\]
which is well-defined on the Hilbert manifold 
\[
    \hat{\mathcal{H}}_{q^\pm} = \{\gamma\in H^1([-1,1];\R^2\setminus\{c_1,c_2\}):\,\gamma(\pm 1)=q^\pm\}.
\]
Looking for complex solutions, we can impose a finer topological constraint on $u$, counting the number of turns it makes around each of the centres. 
Indeed, we can glue to each path $\gamma$ in $\hat{\mathcal{H}}_{q^\pm}$ a path connecting $q^+$ to $q^-$ and check to which homotopy class the resulting path belongs. The result depends of course on the choice of cap; we will always \emph{close} $\gamma$ by adding a circular path on $\partial B_{|q^-|}$ which goes from $q^+$ to $q^-$. Note that, in this section, this is the only point in which we use the requirement \eqref{def:fixed_end}, which can be easily extended to a generic choice of $q^\pm\in\R^2\setminus\{c_1, c_2\}$. 

Since homotopy classes of loops are coded by reduced words we will decompose $\hat{\mathcal{H}}_{q^\pm}$ in connected components labelled by these words. Fix $k\in\N$ and, for any finite $k$-admissible sequence $s_1,\ldots,s_m\in\{\pm1,\pm2\}$ (see Definition \ref{def:admissible_k}), consider the homotopy classes of the loop obtained as 
	\begin{equation}\label{def:homotopy_classes}
		\gamma_{s_1,\ldots,s_m}\uguale \al_{|s_1|}^{\text{sgn}(s_1)}\al_{|s_2|}^{\text{sgn}(s_2)}\cdots \al_{|s_m|}^{\text{sgn}(s_m)}
	\end{equation}
Closing each path $\gamma$ in $\hat{\mathcal{H}}_{q^\pm}$ and checking wether it belongs to homotopy class $[\gamma_{s_1, \ldots, s_m}]$ we obtain the desired decomposition. 
The \emph{closure} of a fixed-end path $\gamma$ obtained following the procedure outlined abovez is denoted as $\phi_\gamma$. So, we can define
\[
\hat{\mathcal{H}}_{q^\pm}(s_1,\ldots,s_m) = \left\lbrace \gamma\in \hat{\mathcal{H}}_{q^\pm}:\, \phi_\gamma \in [\gamma_{s_1,\ldots,s_m}]\right\rbrace
\]
and its weak $H^1$ closure $\mathcal{H}_{q^\pm}(s_1,\ldots,s_m)$ (see Figure \ref{fig:fixed_end_example}; for a rigorous definition of $\phi_\gamma$ see \cite{SoaTer2012,BosDamPap2018}). At this point, we can prove the following result:
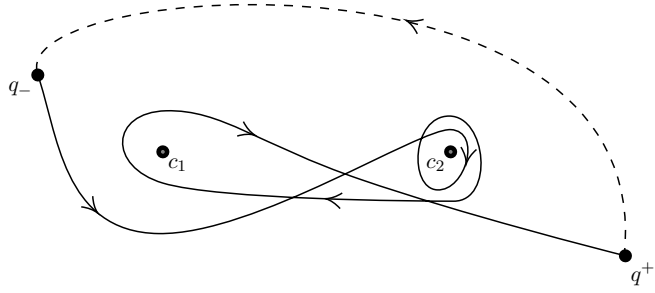
\begin{figure}[t]
    \centering
	\resizebox{.5\textwidth}{!}{\tikzset{every picture/.style={line width=0.75pt}} 

\begin{tikzpicture}[x=0.75pt,y=0.75pt,yscale=-1,xscale=1]

\draw  [fill={rgb, 255:red, 128; green, 128; blue, 128 }  ,fill opacity=1 ][line width=2.25]  (228.19,180.26) .. controls (227.88,178.75) and (228.86,177.29) .. (230.38,177) .. controls (231.9,176.72) and (233.38,177.71) .. (233.69,179.23) .. controls (233.99,180.75) and (233.01,182.21) .. (231.49,182.49) .. controls (229.97,182.78) and (228.49,181.78) .. (228.19,180.26) -- cycle ;
\draw  [fill={rgb, 255:red, 128; green, 128; blue, 128 }  ,fill opacity=1 ][line width=2.25]  (419.19,180.26) .. controls (418.88,178.75) and (419.86,177.29) .. (421.38,177) .. controls (422.9,176.72) and (424.38,177.71) .. (424.69,179.23) .. controls (424.99,180.75) and (424.01,182.21) .. (422.49,182.49) .. controls (420.97,182.78) and (419.49,181.78) .. (419.19,180.26) -- cycle ;
\draw  [fill={rgb, 255:red, 0; green, 0; blue, 0 }  ,fill opacity=1 ][line width=2.25]  (145.19,129.26) .. controls (144.88,127.75) and (145.86,126.29) .. (147.38,126) .. controls (148.9,125.72) and (150.38,126.71) .. (150.69,128.23) .. controls (150.99,129.75) and (150.01,131.21) .. (148.49,131.49) .. controls (146.97,131.78) and (145.49,130.78) .. (145.19,129.26) -- cycle ;
\draw  [fill={rgb, 255:red, 0; green, 0; blue, 0 }  ,fill opacity=1 ][line width=2.25]  (535.19,249.26) .. controls (534.88,247.75) and (535.86,246.29) .. (537.38,246) .. controls (538.9,245.72) and (540.38,246.71) .. (540.69,248.23) .. controls (540.99,249.75) and (540.01,251.21) .. (538.49,251.49) .. controls (536.97,251.78) and (535.49,250.78) .. (535.19,249.26) -- cycle ;
\draw    (147.94,128.75) .. controls (161,166) and (164,233) .. (228,234) .. controls (292,235) and (401,164) .. (422,165) .. controls (443,166) and (432,206) .. (413,205) .. controls (394,204) and (396,156) .. (421,156) .. controls (446,156) and (451,212.5) .. (424,212.5) .. controls (397,212.5) and (268,211.5) .. (231,200.5) .. controls (194,189.5) and (196,153.5) .. (233,152.5) .. controls (270,151.5) and (294.39,174.81) .. (362,198) .. controls (429.61,221.19) and (521.94,244.75) .. (537.94,248.75) ;
\draw   (184.4,209.16) .. controls (184.64,213.69) and (185.73,217.42) .. (187.67,220.36) .. controls (184.88,218.21) and (181.25,216.84) .. (176.76,216.26) ;
\draw   (439.1,178.2) .. controls (435.69,181.18) and (433.42,184.35) .. (432.31,187.69) .. controls (432.28,184.16) and (431.11,180.46) .. (428.79,176.57) ;
\draw   (284.39,159.69) .. controls (286.44,163.72) and (288.96,166.69) .. (291.92,168.59) .. controls (288.5,167.76) and (284.62,167.98) .. (280.28,169.27) ;
\draw   (348.59,217.41) .. controls (345.53,214.07) and (342.32,211.87) .. (338.96,210.83) .. controls (342.48,210.72) and (346.16,209.48) .. (350,207.08) ;
\draw  [dash pattern={on 4.5pt off 4.5pt}]  (147.38,126) .. controls (135,65.5) and (556,37.5) .. (537.38,246) ;
\draw   (399.81,101.41) .. controls (397.79,97.35) and (395.3,94.36) .. (392.35,92.44) .. controls (395.77,93.3) and (399.65,93.11) .. (404,91.85) ;

\draw (145.94,131.75) node [anchor=north east] [inner sep=0.75pt]   [align=left] {$\displaystyle q_{-}$};
\draw (539.94,251.75) node [anchor=north west][inner sep=0.75pt]   [align=left] {$\displaystyle q^{+}$};
\draw (232.94,182.75) node [anchor=north west][inner sep=0.75pt]   [align=left] {$\displaystyle c_{1}$};
\draw (419.94,182.75) node [anchor=north east] [inner sep=0.75pt]   [align=left] {$\displaystyle c_{2}$};

\end{tikzpicture}}
	\caption{An example of fixed-end path in $\hat{\mathcal{H}}_{q^\pm}(1,-2, -2, -1, 2)$.}\label{fig:fixed_end_example}
\end{figure}
\begin{theorem}\label{thm:collisionless_fixed_end}
    Let $q^-, q^+$ as in \eqref{def:fixed_end}, let $h$ satisfying \eqref{hyp:energy} and $k\in\N$. Then, for any finite $k$-admissible sequence $s_1,\ldots,s_m\in\{\pm1,\pm2\}$, there exists a collision-free minimiser of the Maupertuis functional $\mathcal{M}_h$ in $\hat{\mathcal{H}}_{q^\pm}(s_1,\ldots,s_m)$. 
    
    Any such minimiser $\eta$, up to reparametrisations, is a solution of \eqref{eq:motion} at energy $h$, with fixed-ends $q^\pm$ such that $\phi_\eta$ belongs to the homotopy class $[\gamma_{s_1,\ldots,s_m}]$.
\end{theorem}
\begin{proof}
    \emph{Step 1 (Existence of minimisers of $\mathcal{M}_h$):} The proof employs direct methods in the calculus of variations. Fix $k\in\N$ and $q^-, q^+$ as in \eqref{def:fixed_end}. Consider a $k$-admissible sequence $s_1,\ldots, s_m\in\{\pm1,\pm2\}$. We show that $\mathcal{M}_h$ is coercive and sequentially weakly lower semi-continuous in the closed space $\mathcal{H}_{q^\pm}(s_1,\ldots,s_m)$. To prove coercivity, we take a sequence $(\gamma_n)$ in $\mathcal{H}_{q^\pm}(s_1,\ldots,s_m)$ such that $\|\gamma_n\|_{H^1}\to+\infty$. Note that, thanks to \eqref{hyp:energy}, $\int_{-1}^1\left[V(\gamma_n)+h\right]\ge C>0$ for any $n\in\N$. So, if $\|\dot{\gamma}_n\|_2\to+\infty$, there's nothing to prove since
    \[
        \mathcal{M}_h(\gamma_n)=\frac12\int_{-1}^1\lvert\dot{\gamma}_n(t)\rvert^2\,dt\int_{-1}^{1}\left[V(\gamma_n(t))+h\right]\,dt\ge C\|\dot{\gamma}_n\|_2^2.
    \]
    Assume instead that $\|\dot{\gamma}_n\|_2\le C_1$ for any $n\in\N$ and some constant $C_1>0$. For any $t\in[-1,1]$, from the fundamental theorem of calculus and H\"older inequality we have that 
    \[
        |\gamma_n(t)|\le |q^-| + \int_{-1}^1 |\dot{\gamma}_n(\tau)|\,d\tau \le |q^-| + \sqrt{2}\|\dot{\gamma}_n\|_2
    \]
    and so $\|\gamma_n\|_{\infty}$ is uniformly bounded by $C_2 = |q^-| + \sqrt{2}C_1$. This is a contradiction since $\|\gamma_n\|_2\le\sqrt{2}\|\gamma_n\|_{\infty}$ and the $H^1$ norm is diverging with $n$. 
    
    We observe that since we are studying a fixed-end problem for a path which belongs to a non-trivial homotopy class described by the sequence $s_1,\ldots,s_m$, we have that $\|\dot{\gamma}_n\|_2\ge C>0$ uniformly. Moreover, we have already seen that the potential part of the Maupertuis functional is bounded from below by a positive constant. Since the product of two positive and lower semi-continuous functions is a lower semi-continuous function, the proof of the existence is now concluded. 

    \emph{Step 2 (Minimisers are collision-less):} Let $\gamma$ be a minimiser of $\mathcal{M}_h$ in $\mathcal{H}_{q^\pm}(s_1,\ldots,s_m)$. Assume by contradiction that there exists $t^*\in(-1,1)$ such that $\gamma(t^*)=c_1$ (if $\gamma$ collides with $c_2$ the proof is the same). Let $a\ge -1$ be such that $\gamma(t)\neq c_1,c_2$ for any $t\in[a,t^*)$ and, without loss of generality, assume that $\gamma([a, t^*])\subset B_r(c_1)$, where $r>0$ has been defined in \eqref{eq:potential}. In this way, for any $t\in[a,t^*)$, the Maupertuis functional reads
    \begin{equation}\label{eq:maup_step}
        \mathcal{M}_h(\gamma|_{[a,t]}) = \frac12\int_a^t|\dot{\gamma}(\tau)|^2\,d\tau\int_a^t\left[\frac{f_1(\gamma(\tau))}{|\gamma(\tau)-c_1|^2} + W_1(\gamma(\tau)) + h\right]\,d\tau
    \end{equation}
    and we can assume that there exists a positive constant $C$ such that $f_1(x)\ge \frac{1}{C}$ for any $x\in\gamma([a,t^*])$. From the fundamental theorem of calculus, for any $t\in[a,t^*]$ we have that
    \[
        \log\lvert \gamma(t)-c_1\rvert - \log\lvert \gamma(a) - c_1\rvert = \int_a^t\frac{\langle\gamma(\tau) - c_1, \dot{\gamma}(\tau)\rangle}{|\gamma(\tau)-c_1|^2}\,d\tau.
    \]
    Using Cauchy-Schwarz and H\"older inequality, and then the estimate from below on $f_1$ we have that
    \[
        \begin{aligned}
        \lvert \log |\gamma(t)-c_1| - \log |\gamma(a)-c_1|\rvert^2 \le & \left(\int_a^t\frac{|\dot{\gamma}(\tau)|}{|\gamma(\tau)-c_1|}\,d\tau\right)^2 \\
            &\le \int_{a}^t|\dot{\gamma}(\tau)|^2\,d\tau\int_{a}^t\frac{1}{|\gamma(\tau)-c_1|^2}\,d\tau \\
            &\le C\int_a^t |\dot{\gamma}(\tau)|^2\,d\tau\int_a^t\frac{f_1(\gamma(\tau))}{|\gamma(\tau)-c_1|^2}\,d\tau.
        \end{aligned}
    \]
    At this point, from \eqref{eq:maup_step}, we have that
    \[
    C\mathcal{M}_h(\gamma|_{[a,t]}) = \frac{C}{2}\int_a^t |\dot{\gamma}(\tau)|^2\,d\tau\int_a^t\frac{f_1(\gamma(\tau))}{|\gamma(\tau)-c_1|^2}\,d\tau + \frac{C}{2}\int_a^t |\dot{\gamma}(\tau)|^2\,d\tau\int_a^t\left[W_1(\gamma(\tau))+h\right]\,d\tau
    \]
    and so 
    \[
    \begin{aligned}
        \frac12\lvert\log |\gamma(t)-c_1| - \log |\gamma(a)-c_1|\rvert^2+\frac{C}{2}\int_a^t|\dot{\gamma}(\tau)|^2\,d\tau\int_a^t\left[W_1(\gamma(\tau))+h\right]\,d\tau &\le \frac{C}{2}\mathcal{M}_h(\gamma|_{[a,t]}) \\
        &\le \frac{C}{2}\mathcal{M}_h(\gamma).
    \end{aligned}
    \]
    Taking the limit as $t\to t^*$ in the previous equation, we obtain that $\mathcal{M}_h(\gamma)=+\infty$. Since the class $\mathcal{H}_{q^\pm}(s_1,\ldots,s_m)$ contains collision-free representatives for which the Maupertus functional is finite,
    the infimum of $\mathcal{M}_h$ is finite on the class. On the other hand, the previous estimate shows that any curve $\gamma$ developing a collision must satisfy $\mathcal{M}_h(\gamma)=+\infty$. Therefore, minimisers are collision-free and belong to $\hat{\mathcal{H}}_{q^\pm}(s_1,\ldots,s_m)$. 
\end{proof}

\subsection{Periodic solutions}\label{sec:periodic}

In the previous section we proved Theorem \ref{thm:collisionless_fixed_end} which estabilshes the existence of classical solutions of \eqref{eq:motion}-\eqref{eq:energy} with fixed-ends and satisfying a precise topological constraint determined by an admissible sequence $s_1,\ldots,s_m\in\{\pm1,\pm2\}$ (see Definition \ref{def:admissible_k}). The present section focuses on periodic solutions of \eqref{eq:motion}-\eqref{eq:energy}, whose images on the configuration space are closed curves. For this reason, we firstly obtain a more general existence result for any non-trivial free homotopy class of loops in $\R^2\setminus\{c_1,c_2\}$, and then we characterize this result in relation with the definition of admissible sequences $s_1,\ldots,s_m$.  
\begin{definition}
    For any non-contractible free homotopy class $[\tau]$ of $\R^2\setminus\{c_1,c_2\}$ we define
    \[
    \hat{\mathcal{H}}(\tau)=\{\gamma\in H^1(\cerchio;\R^2\setminus\{c_1,c_2\}):\,\gamma\in[\tau]\}
    \]
    and its weak $H^1$ closure $\mathcal{H}(\tau)$.
\end{definition}
The next result establishes the existence of minimisers in $\mathcal{H}(\tau)$ for the Maupertuis functional and its proof follows the same arguments of the first step of Theorem \ref{thm:collisionless_fixed_end}.
\begin{proposition}
    Let $[\tau]$ be a free homotopy class of $\R^2\setminus\{c_1,c_2\}$ different from the trivial one and from $[\alpha^k_i]$,  for $i=1,2$ and $k\in \mathbb{Z}$. For any $h$ satisfying \eqref{hyp:energy}, the Maupertuis functional $\mathcal{M}_h$ achieves its minimum in $\mathcal{H}(\tau)$. 
\end{proposition}

The following proposition describes a property of Maupertuis minimisers: if $\gamma$ is a minimiser of $\mathcal{M}_h$ in $\mathcal{H}(\tau)$, whenever we select a sub-arc of $\gamma$, we obtain a fixed-end minimiser (for a proof of this fact see \cite[Lemma 3.5]{BarCan}).
\begin{proposition}\label{prop:local_minimisers}
    Let $\gamma$ be a minimiser of $\mathcal{M}_h$ in $\mathcal{H}(\tau)$, $[a,b]\subset \cerchio$ and $q^-=\gamma(a)$, $q^+=\gamma(b)$. Then, the path $\gamma|_{[a,b]}$ is a minimiser in the space
        \[
        \mathcal{H}_{q^\pm}(\tau) = \left\lbrace \phi\in H^1([a,b];\R^2):\,\phi(a)=q^-,\ \phi(b)=q^+,\ \phi\#\gamma|_{\cerchio\setminus[a,b]}\in\mathcal{H}(\tau)\right\rbrace.
        \]
\end{proposition}

Using the same argument of the previous section, we can obtain the following theorem concerning the existence of periodic solutions realising given homotopy classes. Indeed, the space $\mathcal{H}_{q^\pm}(\tau)$ introduced in the previous proposition corresponds to a fixed-end homotopy class determined by a finite admissible sequence $s_1,\ldots,s_m\in\{\pm1,\pm2\}$ and so we can refer to the proof of Theorem \ref{thm:collisionless_fixed_end}. We recall that $\al_1$ and $\al_2$ denote two closed curves parametrising the disks enclosing $c_1$ and $c_2$, respectively. 
\begin{theorem}\label{thm:collisionless_tau}
    Assume that $[\tau]$ is a non-contractible free homotopy class of $\R^2\setminus\{c_1,c_2\}$ which is not represented by a power of a single generator $\al_i$. Then, for any $h$ satisfying \eqref{hyp:energy}, there exists a collision-less minimiser of $\mathcal{M}_h$ in $\hat{\mathcal{H}}(\tau)$. In other words, equation \eqref{eq:motion} admits a classical periodic solution at energy $h$ which belongs to $[\tau]$.    
\end{theorem}
\begin{proof}
		The proof is an adaptation of the one given in Theorem \ref{thm:collisionless_fixed_end}. It is enough to note that $\mathcal{M}_h$ is again coercive and lower-semicontinuous and the argument to exclude collisions is local. The statement follows then by an application of the direct method.
    \end{proof}

\begin{remark}
    In the statement of the previous theorem, we have excluded the homotopy classes represented as powers of a single generator $\al_i$. Assume that $c_1=0$, fix $m\in\Z$ and define the one-parameter family of loops $\gamma_\ve\colon\cerchio\to\R^2$ as $\gamma_\ve(t)=\ve e^{2\pi i m t}$, where we have identified $\cerchio\simeq\R/\Z$. Each $\gamma_\ve$ is a circular loop winding $m$ times around the origin and thus belongs to the homotopy class $[\al_1^m]$. A direct computation yields
    \[
    \begin{aligned}
        \mathcal{M}_h(\gamma_\ve) &= \frac12 \int_{\cerchio} 4\pi^2m^2\ve^2\,dt \int_{\cerchio} \left[ \frac{f_1(\gamma_\ve(t))}{\ve^2} + W_1(\gamma_\ve(t)) + h\right]\,dt \\
        &= 2\pi^2m^2\int_{\cerchio} f_1(\gamma_\ve(t))\,dt + o(1),
    \end{aligned}
    \]
    as $\ve\to 0^+$. Since the latter quantity remains uniformly bounded as $\ve\to 0^+$, the strong force argument used in the proof of Theorem \ref{thm:collisionless_fixed_end} does not exclude collisions for loops winding around a single centre. 
\end{remark}
The following result is a corollary of Theorem \ref{thm:collisionless_tau}, stated with the same notations of Theorem \ref{thm:collisionless_fixed_end}. It guarantees the existence of collision-less minimisers in the homotopy classes $[\gamma_{s_1,\ldots,s_m}]$ defined in \eqref{def:homotopy_classes}, which are the ones we are interested in. We then introduce the following notation for the homotopy classes $[\gamma_{s_1,\ldots,s_m}]$
\[
    \mathcal{H}(s_1\,\ldots,\,s_m) = \mathcal{H}(\gamma_{s_1,\ldots,s_m}),\quad \hat{\mathcal{H}}(s_1\,\ldots,\,s_m) = \hat{\mathcal{H}}(\gamma_{s_1,\ldots,s_m}).
\]
\begin{corollary}\label{cor:collisionless}
    For any $h$ satisfying \eqref{hyp:energy}, $k\in\N$ and any finite $k$-admissible sequence $s_1,\ldots,s_m\in\{\pm1,\pm2\}$, there exists a collisionless minimiser of $\mathcal{M}_h$ in $\hat{\mathcal{H}}(s_1,\ldots,s_m)$. In other words, equation \eqref{eq:motion} admits a classical periodic solution at energy $h$ which is represented as a loop in $[\gamma_{s_1,\ldots,s_m}]$ (see \eqref{def:homotopy_classes}). 
\end{corollary}

\section{A compactness result on minimisers}\label{sec:compactness}

In this section, we establish two qualitative properties of the minimisers of the Maupertuis functional. First, we prove that periodic minimisers cannot escape to infinity; an analogous property also holds for suitable restrictions of fixed-end minimisers. Moreover, we show that, if the maximum number of turns around the centres is fixed to some $k\in\N$, then the minimisers of the Maupertuis functional remain at a bounded distance from the centres.

\subsection{Minimisers cannot escape to infinity}

The first result we provide is independent of $k$ and concerns the confinement of minimisers within a bounded region $\mathcal{D}\subset\R^2$. To this end, we recall some qualitative results which describe intersection properties of minimisers; these will be used below. For precise statements and further details we refer to \cite{HasSco1985,BarCan,Cas2017}.  
\begin{proposition}\label{prop:taut_minimisers}
    Let $h$ satisfy \eqref{hyp:energy}, and let $[\tau]$ be a non-trivial free homotopy class of $\R^2\setminus\{c_1,c_2\}$. If $\gamma$ is a minimiser of $\mathcal{M}_h$ in $\mathcal{H}(\tau)$, then:
    \begin{enumerate}[label = \roman*)]
        \item $\gamma$ has no singular 1-gons in $\R^2\setminus\{c_1,c_2\}$;
        \item $\gamma$ has no singular 2-gons in $\R^2\setminus\{c_1,c_2\}$. 
    \end{enumerate}
    As a consequence, minimisers are \emph{taut} in their homotopy classes. 
\end{proposition}
\begin{proof} 
    Assertion $i)$ follows from the super-additivity $\mathcal{M}_h$: removing any contractible component of $\gamma$ would yield a strictly smaller value of the functional, contradicting the minimality of $\gamma$. Assertion $ii)$ follows by regularity of collision-less minimisers (see Proposition \ref{prop:maupertuis}): an easy variation on $\gamma$ would otherwise produce another non-smooth minimiser. Theorem \ref{thm:taut} then implies that minimisers are taut in their homotopy classes. 
\end{proof}
The next result states that minimisers belonging to distinct homotopy classes minimise the number of mutual intersections.
\begin{proposition}\label{prop:intersection_minimisers}
    Let $h$ satisfy \eqref{hyp:energy}, and let $\gamma_1,\gamma_2$ be two minimisers of the Maupertuis functional belonging to two distinct non-trivial free homotopy classes of $\R^2\setminus\{c_1,c_2\}$. Then: 
    \begin{enumerate}[label = \roman*)]
        \item every intersection between $\gamma_1$ and $\gamma_2$ is transversal;
        \item the pair $\gamma_1$ and $\gamma_2$ does not form any singular $2$-gons in $\R^2\setminus\{c_1,c_2\}$. 
    \end{enumerate}
    In other words, according to Definition \ref{def:taut}, $\gamma_1$ and $\gamma_2$ are in \emph{minimal position}. 
\end{proposition}
\begin{proof}
    Since $\gamma_1$ and $\gamma_2$ are collisionless minimisers, they solve the same differential equation \eqref{eq:motion}; thus assertion $i)$ follows directly. Assertion $ii)$ is a direct consequence of the regularity of minimisers and the definition of singular 2-gons.
\end{proof}

We now prove the main result of this section. Consider the homotopy class of a disk enclosing both centres $c_1$ and $c_2$. Using the notations introduced in \eqref{def:homotopy_classes}, this class is $[\al_1\al_2]$ and the Maupertuis functional admits a collision-less minimiser $\varphi$ in such class (cf. Corollary \ref{cor:collisionless}). By Proposition \ref{prop:taut_minimisers}, $\vp$ is a simple curve and hence bounds a topological disk $\mathcal{D}$ which contains both centres $c_1$ and $c_2$. The following result shows that $\mathcal{D}$ contains all other periodic minimisers of the Maupertuis functional with the only possible exception of other minimizers in $[\alpha_1\alpha_2]$.

\begin{lemma}\label{lem:boundedness_minimisers}
    Let $h$ satisfy \eqref{hyp:energy}. Then: 
    \begin{enumerate}[label = \roman*)]
        \item for any non-trivial and different from $[\alpha_1\alpha_2]$ free homotopy class $[\tau]$ of $\R^2\setminus\{c_1,c_2\}$, the support of every minimiser of $\mathcal{M}_h$ in $\hat{\mathcal{H}}(\tau)$ is entirely contained in $\mathcal{D}$;
        \item if $q^\pm\in{\mathcal{D}}\setminus\{c_1,c_2\}$, then for any finite sequence $s_1,\ldots,s_m\in\{\pm1,\pm2\}$, the support of every minimiser of $\mathcal{M}_h$ in $\hat{\mathcal{H}}_{q^\pm}(s_1,\ldots,s_m)$ is entirely contained in $\bar{\mathcal{D}}$;
        \item if $q^\pm\in\R^2\setminus\mathcal{D}$, then, for any finite alternating sequence $s_1,\ldots,s_m$ of $(1,2)$ or $(-1,-2)$, the support of every minimiser of $\mathcal{M}_h$ in $\hat{\mathcal{H}}_{q^\pm}(s_1,\ldots,s_m)$ lies entirely outside $\bar{\mathcal{D}}$. If $q^\pm\in\partial\mathcal{D}$, then the support of every minimiser in $\hat{\mathcal{H}}_{q^\pm}(s_1,\ldots,s_m)$ is an arc of $\partial\mathcal{D}$;
        \item if $q^\pm\in\R^2\setminus\mathcal{D}$, and $s_1,\ldots,s_m$ is not one of the sequences in the previous point, then there exists a subinterval $[a,b]\subset J$ such that the support of the restriction $\gamma|_{[a,b]}$ of every minimiser $\gamma$ in $\hat{\mathcal{H}}_{q^\pm}(s_1,\ldots,s_m)$ is entirely contained in $\bar{\mathcal{D}}$.
    \end{enumerate}
\end{lemma}
\begin{figure}[t]
    \centering
	\resizebox{.75\textwidth}{!}{\input{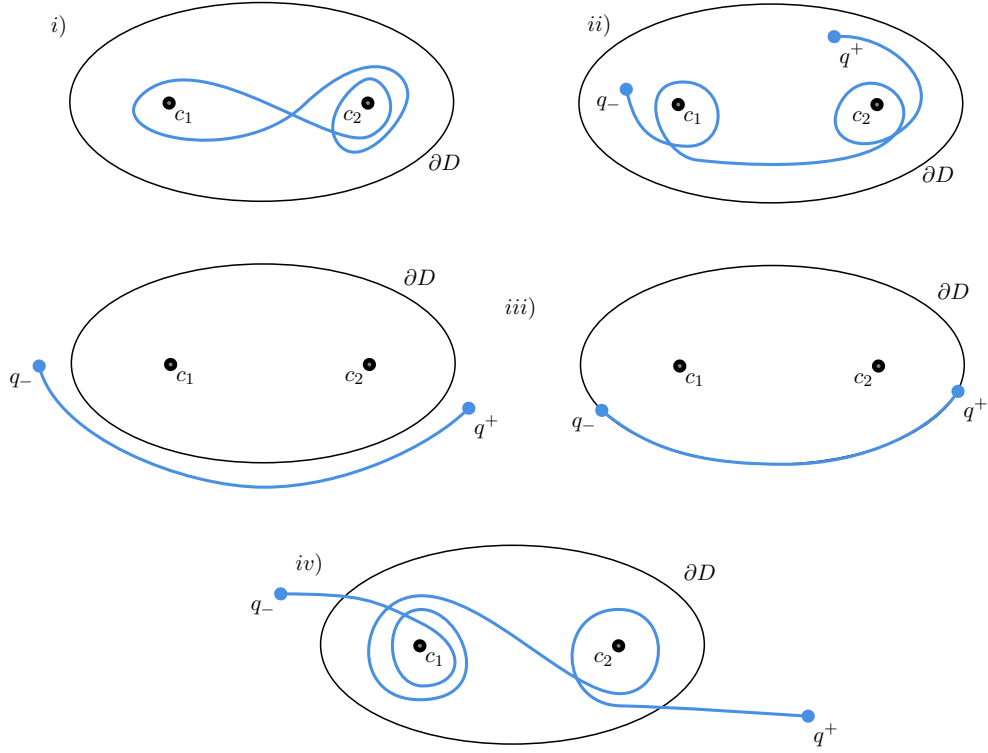}}
	\caption{The four possibilities described in Lemma \ref{lem:boundedness_minimisers}}\label{fig:boundedness}
\end{figure}
\begin{proof}
    Fix $h$ satisfying \eqref{hyp:energy} and let $\varphi$ be a minimiser of $\mathcal{M}_h$ in the homotopy class $[\al_1\al_2]$ which bounds the topological disk $\mathcal{D}$.  To prove $i)$, assume by contradiction that a minimiser in $\mathcal{H}(\tau)$ has one transversal intersection with $\varphi$. Then, it must intersect $\varphi$ at least twice, thereby generating a singular 2-gon and contradicting Proposition \ref{prop:intersection_minimisers} (see Figure \ref{fig:bounded_minimisers}). Assertions $ii)$ and $iii)$ follow by analogous reasoning using Proposition \ref{prop:intersection_minimisers}. Finally, $iv)$ is a direct consequence of Proposition \ref{prop:taut_minimisers}. 

\end{proof}
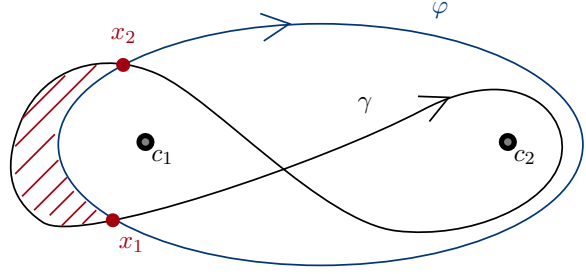
\begin{figure}[t]
    \centering
	\resizebox{.45\textwidth}{!}{\tikzset{every picture/.style={line width=0.75pt}} 

\begin{tikzpicture}[x=0.75pt,y=0.75pt,yscale=-1,xscale=1]

\draw  [fill={rgb, 255:red, 128; green, 128; blue, 128 }  ,fill opacity=1 ][line width=2.25]  (146.86,131.18) .. controls (146.46,129.18) and (147.76,127.25) .. (149.77,126.87) .. controls (151.78,126.5) and (153.73,127.81) .. (154.14,129.82) .. controls (154.54,131.82) and (153.24,133.75) .. (151.23,134.13) .. controls (149.22,134.5) and (147.27,133.19) .. (146.86,131.18) -- cycle ;
\draw  [fill={rgb, 255:red, 128; green, 128; blue, 128 }  ,fill opacity=1 ][line width=2.25]  (356.86,131.18) .. controls (356.46,129.18) and (357.76,127.25) .. (359.77,126.87) .. controls (361.78,126.5) and (363.73,127.81) .. (364.14,129.82) .. controls (364.54,131.82) and (363.24,133.75) .. (361.23,134.13) .. controls (359.22,134.5) and (357.27,133.19) .. (356.86,131.18) -- cycle ;
\draw  [color={rgb, 255:red, 2; green, 54; blue, 118 }  ,draw opacity=1 ] (100,132) .. controls (100,93.34) and (168.05,62) .. (252,62) .. controls (335.95,62) and (404,93.34) .. (404,132) .. controls (404,170.66) and (335.95,202) .. (252,202) .. controls (168.05,202) and (100,170.66) .. (100,132) -- cycle ;
\draw    (79,117) .. controls (92,88) and (125,76.5) .. (157.75,91.25) .. controls (190.5,106) and (253,177) .. (296,182) .. controls (339,187) and (384,167) .. (391,141) .. controls (398,115) and (368,82) .. (311,112) .. controls (254,142) and (112,191) .. (91,177) .. controls (70,163) and (68,143) .. (79,117) -- cycle ;
\draw  [color={rgb, 255:red, 163; green, 0; blue, 16 }  ,draw opacity=1 ][fill={rgb, 255:red, 163; green, 0; blue, 16 }  ,fill opacity=1 ][line width=2.25]  (129.09,176.32) .. controls (128.82,174.94) and (129.71,173.62) .. (131.09,173.36) .. controls (132.47,173.1) and (133.81,174.01) .. (134.09,175.38) .. controls (134.36,176.76) and (133.47,178.08) .. (132.09,178.34) .. controls (130.71,178.6) and (129.37,177.7) .. (129.09,176.32) -- cycle ;
\draw  [color={rgb, 255:red, 163; green, 0; blue, 16 }  ,draw opacity=1 ][fill={rgb, 255:red, 163; green, 0; blue, 16 }  ,fill opacity=1 ][line width=2.25]  (135.09,86.32) .. controls (134.82,84.94) and (135.71,83.62) .. (137.09,83.36) .. controls (138.47,83.1) and (139.81,84.01) .. (140.09,85.38) .. controls (140.36,86.76) and (139.47,88.08) .. (138.09,88.34) .. controls (136.71,88.6) and (135.37,87.7) .. (135.09,86.32) -- cycle ;
\draw  [color={rgb, 255:red, 2; green, 54; blue, 118 }  ,draw opacity=1 ] (217.05,56.08) -- (233.94,61.99) -- (219.08,71.95) ;
\draw   (308.4,104.05) -- (326.28,104.67) -- (315.05,118.6) ;
\draw [color={rgb, 255:red, 163; green, 0; blue, 16 }  ,draw opacity=1 ]   (75,149) -- (98,126) ;
\draw [color={rgb, 255:red, 163; green, 0; blue, 16 }  ,draw opacity=1 ]   (79,162) -- (100,141) ;
\draw [color={rgb, 255:red, 163; green, 0; blue, 16 }  ,draw opacity=1 ]   (75,133) -- (122,86) ;
\draw [color={rgb, 255:red, 163; green, 0; blue, 16 }  ,draw opacity=1 ]   (86,170) -- (105,151) ;
\draw [color={rgb, 255:red, 163; green, 0; blue, 16 }  ,draw opacity=1 ]   (95,173) -- (109,159) ;
\draw [color={rgb, 255:red, 163; green, 0; blue, 16 }  ,draw opacity=1 ]   (102,179) -- (116,165) ;
\draw [color={rgb, 255:red, 163; green, 0; blue, 16 }  ,draw opacity=1 ]   (114,178) -- (122,170) ;
\draw [color={rgb, 255:red, 163; green, 0; blue, 16 }  ,draw opacity=1 ]   (79,117) -- (104,92) ;

\draw (152.5,133.5) node [anchor=north west][inner sep=0.75pt]   [align=left] {$\displaystyle c_{1}$};
\draw (362.5,133.5) node [anchor=north west][inner sep=0.75pt]   [align=left] {$\displaystyle c_{2}$};
\draw (315,48) node [anchor=north west][inner sep=0.75pt]  [color={rgb, 255:red, 2; green, 54; blue, 118 }  ,opacity=1 ] [align=left] {$\displaystyle \varphi $};
\draw (133.59,184) node [anchor=north west][inner sep=0.75pt]  [color={rgb, 255:red, 163; green, 0; blue, 16 }  ,opacity=1 ] [align=left] {$\displaystyle x_{1}$};
\draw (129,63) node [anchor=north west][inner sep=0.75pt]  [color={rgb, 255:red, 163; green, 0; blue, 16 }  ,opacity=1 ] [align=left] {$\displaystyle x_{2}$};
\draw (272,103) node [anchor=north west][inner sep=0.75pt]   [align=left] {$\displaystyle \gamma $};

\end{tikzpicture}}
	\caption{Proof of Lemma \ref{lem:boundedness_minimisers}. Here, $\varphi$ is a minimiser in the homotopy class $[\al_1\al_2]$, while $\gamma$ is a mimimiser in another homotopy class. The dashed red region is bounded by a singular 2-gon, not tolerated by Proposition \ref{prop:intersection_minimisers}. }\label{fig:bounded_minimisers}
\end{figure}

\subsection{Uniform distance from the centres}

We recall that both periodic and fixed-end minimisers are subject to a topological constraint that determines the number of times a minimiser winds around each centre. More precisely, for a fixed $k\in\N$, and for any finite $k$-admissible sequence $s_1,\ldots,s_m\in\{\pm1,\pm2\}$, we consider the homotopy classes of the loop $\gamma_{s_1,\ldots,s_m}$ introduced in \eqref{def:homotopy_classes}.

In this part of the section, we prove that once $k\in\N$ is fixed, the minimisers of the Maupertuis functional remain at a bounded distance from the centres. Such bound depends only on $k$ and on the energy level $h$. To establish this property, we introduce an auxiliary minimisation problem involving only one of the singularities. The argument yields a local result, which then applies directly to the minimisers provided by Theorems \ref{thm:collisionless_fixed_end}-\ref{thm:collisionless_tau} and Corollary \ref{cor:collisionless}.  

Without loss of generality, we assume $c_1 = (0,0)$ and fix $r\in(0, |c_2|/2)$ so that 
\[
V(x) = \frac{f_1(x)}{\lvert x\rvert^2}+W_1(x),\quad \text{for}\ x\in B_r
\]
(cf. \eqref{eq:potential}). For any $q_0,q_1\in\partial B_r$ we consider the space of fixed-ends paths 
\[
    \Gamma_{q_0,q_1}=\left\lbrace \gamma\in H^1([0,1];\bar{B}_r):\,\gamma(0)=q_0,\ \gamma(1)=q_1\right\rbrace.
\]
In order to count the number of turns of a path in $\Gamma_{q_0,q_1}$ performs around the origin, we recall the definition of winding number for a loop, and then we extend it to fixed-end paths in $\Gamma_{q_0,q_1}$ (see Figure \ref{fig:winding}). 

\begin{definition}[Winding number]\label{def:winding_number}
    Let $\gamma\colon\cerchio\to\R^2\setminus\{c\}$ be a loop and identify $\mathbb{R}^2$ with the complex numbers. We say that $\gamma$ has \emph{winding number} $k\in\Z$ with respect to $c$ if 
    \[
    \dfrac{1}{2\pi i}\int_0^1\frac{\gamma'(t)}{\gamma(t)-c}\,dt=\dfrac{1}{2\pi i}\int_\gamma \frac{dw}{w-c}=k.
    \]
    Note that such number is computed with respect to a counterclockwise orientation of $\gamma$.     
\end{definition}

\begin{definition}[Winding number, fixed-end paths]\label{def:winding_number_path}
    For a path in $\gamma\in\Gamma_{q_0,q_1}$ we say that $\gamma$ has winding number $k$ with respect to the origin if the loop obtained by concatenating $\gamma$ with the arc of $\partial B_r$ joining $q_1$ to $q_0$ in counterclockwise direction has winding number $k$ in the sense of Definition \ref{def:winding_number}.
\end{definition}
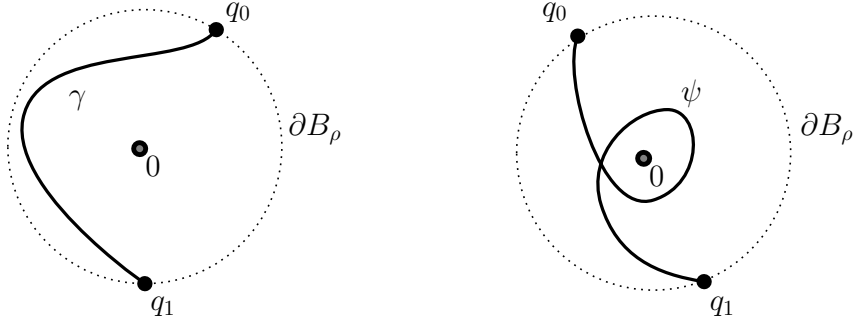
\begin{figure}[t]
    \centering
	\resizebox{.65\textwidth}{!}{\tikzset{every picture/.style={line width=0.75pt}} 

\begin{tikzpicture}[x=0.75pt,y=0.75pt,yscale=-1,xscale=1]

\draw  [color={rgb, 255:red, 0; green, 0; blue, 0 }  ,draw opacity=1 ][dash pattern={on 0.84pt off 2.51pt}] (44.68,155.93) .. controls (44.68,109.09) and (82.66,71.11) .. (129.51,71.11) .. controls (176.36,71.11) and (214.33,109.09) .. (214.33,155.93) .. controls (214.33,202.78) and (176.36,240.76) .. (129.51,240.76) .. controls (82.66,240.76) and (44.68,202.78) .. (44.68,155.93) -- cycle ;
\draw [line width=1.5]    (173.73,83.5) .. controls (157,108) and (63,91) .. (54,139) .. controls (45,187) and (146,252) .. (129.51,240.76) ;
\draw  [color={rgb, 255:red, 0; green, 0; blue, 0 }  ,draw opacity=1 ][fill={rgb, 255:red, 0; green, 0; blue, 0 }  ,fill opacity=1 ][line width=1.5]  (125.87,241.44) .. controls (125.47,239.44) and (126.77,237.51) .. (128.78,237.13) .. controls (130.78,236.76) and (132.74,238.07) .. (133.14,240.08) .. controls (133.55,242.08) and (132.25,244.01) .. (130.24,244.39) .. controls (128.23,244.76) and (126.28,243.45) .. (125.87,241.44) -- cycle ;
\draw  [color={rgb, 255:red, 0; green, 0; blue, 0 }  ,draw opacity=1 ][fill={rgb, 255:red, 0; green, 0; blue, 0 }  ,fill opacity=1 ][line width=1.5]  (170.09,84.19) .. controls (169.69,82.18) and (170.99,80.25) .. (172.99,79.88) .. controls (175,79.5) and (176.96,80.82) .. (177.36,82.82) .. controls (177.77,84.82) and (176.47,86.75) .. (174.46,87.13) .. controls (172.45,87.51) and (170.49,86.19) .. (170.09,84.19) -- cycle ;
\draw  [fill={rgb, 255:red, 128; green, 128; blue, 128 }  ,fill opacity=1 ][line width=2.25]  (122.68,157.86) .. controls (122.27,155.86) and (123.57,153.93) .. (125.58,153.55) .. controls (127.59,153.17) and (129.55,154.49) .. (129.95,156.49) .. controls (130.35,158.5) and (129.05,160.43) .. (127.05,160.8) .. controls (125.04,161.18) and (123.08,159.86) .. (122.68,157.86) -- cycle ;
\draw  [color={rgb, 255:red, 0; green, 0; blue, 0 }  ,draw opacity=1 ][dash pattern={on 0.84pt off 2.51pt}] (359.68,159.93) .. controls (359.68,113.09) and (397.66,75.11) .. (444.51,75.11) .. controls (491.36,75.11) and (529.33,113.09) .. (529.33,159.93) .. controls (529.33,206.78) and (491.36,244.76) .. (444.51,244.76) .. controls (397.66,244.76) and (359.68,206.78) .. (359.68,159.93) -- cycle ;
\draw [line width=1.5]    (397.54,87.5) .. controls (386.97,112.61) and (412.09,199.24) .. (444.33,189) .. controls (476.57,178.76) and (475,133) .. (453,133) .. controls (431,133) and (402,160) .. (412,192) .. controls (422,224) and (446,235) .. (475.73,239.5) ;
\draw  [color={rgb, 255:red, 0; green, 0; blue, 0 }  ,draw opacity=1 ][fill={rgb, 255:red, 0; green, 0; blue, 0 }  ,fill opacity=1 ][line width=1.5]  (393.91,88.19) .. controls (393.5,86.18) and (394.8,84.25) .. (396.81,83.88) .. controls (398.82,83.5) and (400.78,84.82) .. (401.18,86.82) .. controls (401.58,88.82) and (400.28,90.75) .. (398.28,91.13) .. controls (396.27,91.51) and (394.31,90.19) .. (393.91,88.19) -- cycle ;
\draw  [color={rgb, 255:red, 0; green, 0; blue, 0 }  ,draw opacity=1 ][fill={rgb, 255:red, 0; green, 0; blue, 0 }  ,fill opacity=1 ][line width=1.5]  (472.09,240.19) .. controls (471.69,238.18) and (472.99,236.25) .. (474.99,235.88) .. controls (477,235.5) and (478.96,236.82) .. (479.36,238.82) .. controls (479.77,240.82) and (478.47,242.75) .. (476.46,243.13) .. controls (474.45,243.51) and (472.49,242.19) .. (472.09,240.19) -- cycle ;
\draw  [fill={rgb, 255:red, 128; green, 128; blue, 128 }  ,fill opacity=1 ][line width=2.25]  (434.68,163.86) .. controls (434.27,161.86) and (435.57,159.93) .. (437.58,159.55) .. controls (439.59,159.17) and (441.55,160.49) .. (441.95,162.49) .. controls (442.35,164.5) and (441.05,166.43) .. (439.05,166.8) .. controls (437.04,167.18) and (435.08,165.86) .. (434.68,163.86) -- cycle ;

\draw (128.31,160.18) node [anchor=north west][inner sep=0.75pt]  [font=\normalsize] [align=left] {\Large $\displaystyle 0$};
\draw (131.51,247.5) node [anchor=north west][inner sep=0.75pt]   [align=left] {\Large $\displaystyle q_{1}$};
\draw (178,65) node [anchor=north west][inner sep=0.75pt]   [align=left] {\Large $\displaystyle q_{0}$};
\draw (440.31,166.18) node [anchor=north west][inner sep=0.75pt]  [font=\normalsize] [align=left] {\Large $\displaystyle 0$};
\draw (374,65) node [anchor=north west][inner sep=0.75pt]   [align=left] {\Large $\displaystyle q_{0}$};
\draw (477.73,247.5) node [anchor=north west][inner sep=0.75pt]   [align=left] {\Large $\displaystyle q_{1}$};
\draw (81,120) node [anchor=north west][inner sep=0.75pt]   [align=left] {\Large $\displaystyle \gamma $};
\draw (460,113) node [anchor=north west][inner sep=0.75pt]   [align=left] {\Large $\displaystyle \psi $};
\draw (217,134) node [anchor=north west][inner sep=0.75pt]   [align=left] {\Large $\displaystyle \partial B_{\rho }$};
\draw (534,134) node [anchor=north west][inner sep=0.75pt]   [align=left] {\Large $\displaystyle \partial B_{\rho }$};

\end{tikzpicture}}
	\caption{Definition \ref{def:winding_number_path}: the path $\gamma$ has winding number 1 with respect to the origin, hence $\gamma\in\Gamma_{q_0,q_1}(1)$; on the right, $\psi\in\Gamma_{q_0,q_1}(2)$.}\label{fig:winding}
\end{figure}
For $k\in\N$, we define the subset
\[
\Gamma_{q_0,q_1}(k)=\left\lbrace \gamma\in\Gamma_{q_0,q_1}:\,\gamma\ \text{has winding number}\ k\ \text{with respect to the origin}\right\rbrace,
\]
and we minimise the Maupertuis functional $\mathcal{M}_h$ within $\Gamma_{q_0,q_1}(k)$. Using an argument entirely analogous to that of Theorem \ref{thm:collisionless_fixed_end}, one proves that minimisers of $\mathcal{M}_h$ in $\Gamma_{q_0,q_1}(k)$ exist for any $k\in\N$ and $h$ satisfying \eqref{hyp:energy}, and that they are collision-less. Moreover, the following result holds.

\begin{lemma}\label{lem:bounded_euclidean_length}
    For fixed $k\in\N$ and $h$ satisfying \eqref{hyp:energy}, the Euclidean length of any minimiser in $\Gamma_{q_0,q_1}(k)$ is uniformly bounded from above, independently on the end-points $q_0,q_1\in\partial B_r$.    
\end{lemma}
\begin{proof}
    Fix $k\in\N$ and $h$ satisfying \eqref{hyp:energy}, and let $\gamma\in\Gamma_{q_0,q_1}(k)$ be a collision-less minimiser of $\mathcal{M}_h$. Since $\gamma$ is collision-less, it also minimises the Jacobi-Maupertuis length
    \[
    \mathcal{L}_h(\gamma)=\int_0^1\lvert\dot\gamma(t)\rvert\sqrt{V(\gamma(t))+h}\,dt.
    \]
    Thanks to \eqref{hyp:energy}, we have that there exists $C>0$ such that $\sqrt{V(\gamma(t))+h}\ge C$ for all $t\in[0,1]$. The Euclidean length of $\gamma$ is $\mathcal{E}(\gamma)=\int_0^1\lvert\dot\gamma(t)\rvert\,dt$ and hence $\mathcal{L}_h(\gamma)\ge C\mathcal{E}(\gamma)$. 
    
    It is also straightforward to give a uniform bound on $\mathcal{L}_h(\gamma)$ for $\gamma\in \Gamma_{q_0,q_1}(k)$. Let $\psi$ be the $H^1$ path obtained by travelling with constant speed along $\partial B_r$ in the counterclockwise direction from $q_0$ to $q_1$, and then continuing along $\partial B_r$ for $k-1$ further turns so that $\psi$ has winding number $k$ with respect to the origin (see Definition \ref{def:winding_number_path}). In this way, $\psi\in\Gamma_{q_0,q_1}(k)$ and the quantity $\mathcal{L}_h(\psi)$ can be bounded from above by the Jacobi-Maupertuis length of the circular loop winding $k$ times on $\partial B_r$, indipendently on $q_0,q_1$. Moreover, since $\gamma$ is a minimiser, we have that $\mathcal{L}_h(\psi)\ge\mathcal{L}_h(\gamma)$ and the proof is completed.
    
\end{proof}

We now prove that for fixed $k\in\N$ and $h$ satisfying \eqref{hyp:energy}, the minimisers in $\Gamma_{q_0,q_1}(k)$ remain at a uniformly positive distance from the origin, independently on the end-points $q_0,q_1$.

\begin{lemma}\label{lem:bound_collisions}
    Let $h$ satisfy \eqref{hyp:energy}, $k\in\N$ and $q_0,q_1\in \partial B_r$. If $\gamma$ is a minimiser of $\mathcal{M}_h$ in $\Gamma_{q_0,q_1}(k)$, there exists a constant $C>0$, depending only on $r$ and $k$, such that
    \[
    \min_{t\in[0,1]}\lvert\gamma(t)\rvert\ge C.
    \]     
\end{lemma}
\begin{proof}
    Fix $h$ satisfying \eqref{hyp:energy} and $k\in\N$. For $q\in\R^2$, we introduce polar coordinates $q=\rho e^{i\vt}=(\rho,\vt)$ so that the Euclidean metric reads $g_e = dq^2=d\rho^2+\rho^2d\vt^2$. In these coordinates, the Jacobi-Maupertuis metric associated with $V$ and $h$ reads 
    \begin{equation}\label{eq:jm_polar}
    \begin{aligned}
    g_h = (h+V(q))dq^2 &= \left(\frac{f_1(\rho,\vt)}{\rho^2}+W_1(\rho,\vt)+h\right)\left(d\rho^2+\rho^2d\vt^2\right) \\
     & = \frac{f_1(\rho,\vt)}{\rho^2}d\rho^2 + f_1(\rho,\vt)d\vt^2 + \left(W_1(\rho,\vt)+h\right)\left(d\rho^2+\rho^2d\vt^2\right).
    \end{aligned}
    \end{equation}
    Now, consider a sequence of end-points $q_0^n, q_1^n\in\partial B_r$, and a sequence of minimisers $\gamma_n\in\Gamma_{q_0^n,q_1^n}(k)$ of the Maupertuis functional $\mathcal{M}_h$. Since $\gamma_n$ is collision-less for any $n$, we have that 
    \[
        a_n = \min_{t\in[0,1]}\lvert\gamma_n(t)\rvert >0,\quad \text{for any}\ n\in\N. 
    \]
    Assume by contradiction that $a_n\to 0^+$. Let $\eta_n$ be the path in $\Gamma_{q_0^n,q_1^n}(k)$ obtained by travelling along $\partial B_r$ in the counterclockwise direction from $q_0^n$ to $q_1^n$ and then continuing along $\partial B_r$ for $k-1$ further turns. We can write $\eta_n(t)=r e^{i\vp_n(t)}$ and we assume that $\dot{\vp}_n\ge 0$, so that $\int_0^1\dot{\vp}_n(t)\,dt=\vp_n(1)-\vp_n(0)\le 2k\pi$. Since $\lvert\eta_n\rvert\equiv r$, from \eqref{eq:jm_polar} we have:
    \[
    g_h(\dot{\eta}_n(t), \dot{\eta}_n(t)) = f_1(r, \vp_n(t))\dot{\vp}_n^2(t) + \left(W_1(r,\vp_n(t))+h\right)r^2\dot{\vp}_n^2(t),
    \]
    and thus, the Jacobi-length functional reads
    \begin{equation}\label{eq:bound_eta}
    \begin{aligned}
    \mathcal{L}_h(\eta_n)&=\int_0^1g_h(\dot{\eta}_n(t), \dot{\eta}_n(t))^{1/2}\,dt \\
                       &\le\int_0^1f_1(r,\vp_n(t))^{1/2}\dot{\vp}_n(t)\,dt+\int_0^1\left(W_1(r,\vp_n(t))+ h\right)^{1/2}r\dot{\vp}_n(t)\,dt \\
                       &\le 2k\pi(C_{f_1} + r C_{W_1}),
    \end{aligned}
    \end{equation}
    where $C_{f_1} = \max_{\bar{B}_r} \sqrt{f_1}>0$ and $C_{W_1}=\max_{\bar{B}_r}\lvert W_1+h\rvert^{1/2}$. 

    At this point, concerning the sequence of minimisers $(\gamma_n)$, we can write $\gamma_n(t)=\zeta_n(t)e^{i\psi_n(t)}$, for some smooth functions $\zeta_n$ and $\psi_n$. Moreover, given $a,b\ge 0$, we have $2\sqrt{a+b}\ge \sqrt{a}-\sqrt{b}$ and $2\sqrt{a+b}\ge\sqrt{a}+\sqrt{b}$. Therefore, we can write:
    \[
    \begin{aligned}
    \mathcal{L}_h(\gamma_n)&=\int_0^1\left(f_1(\zeta_n,\psi_n)\frac{\dot{\zeta}_n^2}{\zeta_n^2}+f_1(\zeta_n,\psi_n)\dot{\psi}_n^2+\left(W_1(\gamma_n)+h\right)\lvert\dot{\gamma}_n\rvert^2\right)^{1/2}\,dt \\
    &\ge \frac12\int_0^1\left(f_1(\zeta_n,\psi_n)\frac{\dot{\zeta}_n^2}{\zeta_n^2}+f_1(\zeta_n,\psi_n)\dot{\psi}_n^2\right)^{1/2}\,dt-C_{W_1}\mathcal{E}(\gamma_n) \\
    &\ge \frac14\int_0^1\sqrt{f_1(\zeta_n,\psi_n)}\frac{\lvert\dot{\zeta}_n\rvert}{\zeta_n}\,dt+\frac14\int_0^1\sqrt{f_1(\zeta_n,\psi_n)}\lvert\dot{\psi}_n\rvert\,dt- C_{W_1}\mathcal{E}(\gamma_n).
    \end{aligned}
    \]
    Now, we recall that $a_n=\min\limits_{t\in[0,1]}\zeta_n(t)>0$.
    We have
    \[
    \int_0^1\frac{\lvert\dot{\zeta}_n(t)\rvert}{\zeta_n(t)}\,dt = \int_{\{t: \dot{\zeta}_n\ge0\}}\frac{\dot{\zeta}_n(t)}{\zeta_n(t)}\,dt- \int_{\{t: \dot{\zeta}_n\le0\}} \frac{\dot{\zeta}_n(t)}{\zeta_n(t)}\,dt\ge2\log(r/a_n).
    \]
    Moreover, we see that 
    \[
    \int_0^1\lvert\dot{\psi}_n(t)\rvert\,dt\ge\lvert\psi_n(1)-\psi_n(0)\rvert\ge 2(k-1)\pi.
    \]
    We then obtain
    \[
        \mathcal{L}_h(\gamma_n)\ge C_1\log(r/a_n)+C_2-\mathcal{E}(\gamma_n),
    \]
    with $C_1>0$ and $C_2\ge 0$ depending only on $k$ and $r$, since $f_1$ is bounded from below in $\bar{B}_r$. Recalling that $\mathcal{E}(\gamma_n)$ is uniformly bounded with respect to the end-points $q_0^n,q_1^n$ (see Lemma \ref{lem:bounded_euclidean_length}) and that $\mathcal{L}_h(\eta_n)\ge\mathcal{L}_h(\gamma_n)$ for any $n\in\N$, we can combine the previous inequality with \eqref{eq:bound_eta} to obtain
    \[
        C_3\ge C_1\log(r/a_n)+C_2,
    \]
    with $C_3>0$ depending only on $k$ and $\rho$. At this point, recalling that $a_n\to 0^+$, we can pass to the limit as $n\to+\infty$ to obtain a contradiction and the lemma is proved. 

\end{proof}

Finally, we deduce the global compactness property for all minimisers of the Maupertuis functional. 
\begin{lemma}\label{lem:compactness_minimisers}
    Let $h$ satisfy \eqref{hyp:energy} and $k\in\N$. There exists a constant $C>0$, depending only on $h$ and $k$, such that:
    
    \begin{itemize}
        \item \emph{(Periodic minimisers)} For any finite $k$-admissible sequence $s_1,\ldots,s_m\in\{\pm1,\pm2\}$, every minimiser $\gamma$ of $\mathcal{M}_h$ in $\mathcal{H}(s_1,\ldots,s_m)$ satisfies
            \[
                \min_{t\in[0,1]}\lvert\gamma(t)-c_i\rvert\ge C,\quad i=1,2. 
            \] 
        \item \emph{(Fixed-end minimisers)} For each pair of endpoints $q^\pm\in\R^2\setminus\{\bar{B}_{c_1}(C)\cup\bar{B}_{c_2}(C)\}$ and any finite $k$-admissible sequence $s_1,\ldots,s_m\in\{\pm1,\pm2\}$, every minimiser of $\mathcal{M}_h$ in $\mathcal{H}_{q^\pm}(s_1,\ldots,s_m)$ satisfies 
        \[
                \min_{t\in[0,1]}\lvert\gamma(t)-c_i\rvert\ge C,\quad i=1,2. 
        \]
        
        In other words, every minimiser remains at a bounded distance from the centres, depending only on $k$ and $h$. 
    \end{itemize}
\end{lemma}
\begin{proof}
    The proof is exactly the same for periodic and fixed-end minimisers. Recalling the definition of the potential $V$ given in \eqref{eq:potential}, we consider a quantity $\rho< r$. Define $\mathcal{B}_1, \mathcal{B}_2$ as the two balls with radius $\rho$, centred respectively at $c_1$ and $c_2$. Fix $h$ satisfying \eqref{hyp:energy}, $k\in\N$, a finite $k$-admissible sequence $s_1\ldots,s_m\in\{\pm1,\pm2\}$ and consider a Maupertuis minimiser $\gamma\in\mathcal{H}(s_1,\ldots,s_m)$. If $\gamma$ does not intersect any of the balls $\mathcal{B}_1$ and $\mathcal{B}_2$ there's nothing to prove. Assume instead that $\gamma$ intersects the boundary of $\mathcal{B}_1$ (the proof for $\mathcal{B}_2$ is completely analogous). Then, by construction, it has to intersect the boundary of $\mathcal{B}_1$ at least once again. Without loss of generality, we can assume that there exists an interval $[a,b]\subset\cerchio$ such that $\gamma((a,b))\subset\mathring{\mathcal{B}}_1$ and we define $q_0=\gamma(a)$ and $q_1=\gamma(b)$. From Proposition \ref{prop:local_minimisers}, $\gamma|_{[a,b]}$ is a minimiser of the Maupertuis functional in the space of $H^1$ paths with a given number of turns around $c_1$ fixed by the sequence $s_1,\ldots,s_m$. Note that $\gamma|_{[a,b]}$ winds around $c_1$ at most $k$ times; if that were not the case, that it would draw at least a singular 1-gon or 2-gon, violating Proposition \ref{prop:taut_minimisers}. Therefore, recalling the notations of Definition \ref{def:winding_number_path} and Lemma \ref{lem:bound_collisions}, $\gamma|_{[a,b]}\in\Gamma_{x_0,x_1}(p)$, for some $p\le k$. By Lemma \ref{lem:bound_collisions}, $\gamma|_{[a,b]}$ keeps a bounded distance from $c_1$, which depends only on $p$ and $\rho$. 

\end{proof}

\section{Symbolic coding of trajectories}\label{sec:coding}

In this section we are going to prove Theorem \ref{thm:coding}. In particular, we need to show that whenever we fix an admissible bi-infinite sequence $(s_i)_{i\in\Z}\in\{\pm1,\pm2\}^\Z$ in the spirit of Definition \ref{def:admissible_k}, there exists a solution $x\colon \R\to\R^2\setminus\{c_1,c_2\}$ of the motion and energy equations \eqref{eq:motion}-\eqref{eq:energy} such that $x$ has bi-infinite itinerary $(s_i)_{i\in\Z}$. Recall that we can associate to suitable trajectories a symbolic itinerary recording the successive intersections with two distinguished oriented curves $\Gamma_1,\Gamma_2\subset\mathcal{D}$ (cf. Definition \ref{def:itinerary}), where $\mathcal{D}\subset\R^2$ is the compact region on which all periodic solutions are confined (see Lemma \ref{lem:boundedness_minimisers}-\ref{lem:compactness_minimisers}). In this way, we have a natural way to associate symbolic sequences to trajectories: the itineraries. 

In the first part of this section we then construct the curves $\Gamma_1, \Gamma_2$. If we choose suitable solutions of \eqref{eq:motion}, the transversality conditions in Definition \ref{def:itinerary} are automatically verified (compare with Proposition \ref{prop:intersection_minimisers}). We then prove Theorem \ref{thm:coding}: every admissible symbolic sequence is realised by a solution of \eqref{eq:motion}-\eqref{eq:energy} defined on the whole $\R$. The proof follows the classical scheme of symbolic dynamics presented in \cite{BarCan}. Given a finite admissible word, we associate with it a corresponding periodic minimiser of the Maupertuis functional. The compactness properties established in Section \ref{sec:compactness} allow us to pass to the limit along sequences of periodic minimisers with increasing periods, thereby obtaining global trajectories realising prescribed bi-infinite itineraries.

\subsection{Collision minimisers}

The curves $\Gamma_1$ and $\Gamma_2$ will be constructed as trajectories asymptotic to collisions with the centres $\mathcal{D}$, which locally minimise the Maupertuis functional. Their geometric properties ensure that every transversal intersection of a trajectory with $\Gamma_1\cup\Gamma_2$ determines a symbol in the alphabet $\{\pm1,\pm2\}$, according to both the curve being crossed and the orientation of the crossing. We start with the following definition. 

\begin{definition}
    Let $h$ satisfy \eqref{hyp:energy}, let $q\in\R^2\setminus\{c_1,c_2\}$ and let $c\in\{c_1,c_2\}$. A curve $\gamma\colon [0, +\infty)\to\R^2\setminus\{c_1,c_2\}$ satisfying the energy equation \eqref{eq:energy} is called a \emph{collision minimiser} if 
    \begin{itemize}
        \item $\gamma(0)=q$ and $\lim\limits_{t\to+\infty}\gamma(t)=c$;
        \item for every compact interval $[a,b]\subset [0,+\infty)$, the restriction $\gamma|_{[a,b]}$ is a collision-less minimiser of $\mathcal{M}_h$ among all $H^1$ paths joining $\gamma(a)$ and $\gamma(b)$, and homotopic to the segment $[\gamma(a), \gamma(b)]$. 
    \end{itemize}
\end{definition}

So, we will consider collision minimisers with a specific homotopy constraint: we require that, if $\gamma$ connects $q$ and $c$, then it is homotopic to the segment joining $q$ to $c$. This is well defined unless the other center belongs to the segment $[q,c]$; in this case we can slightly perturb the segment to obtain a non collisional curve and consider the relative homotopy class. Now we detect a qualitative property of collision minimisers which will be useful in the proof of their existence. 
\begin{proposition}\label{prop:eight}
    Let $h$ satisfy \eqref{hyp:energy} and let $\gamma_{eight}$ be the eight-shaped minimiser of $\mathcal{M}_h$ in the space $\hat{\mathcal{H}}(1, -2)$. Let $U_{eight}$ be the unbounded connected component of $\R^2\setminus\gamma_{eight}$ and let $q\in U_{eight}$. Then, every collision minimiser at energy $h$ joining $q$ to one of the centres and satisfying the homotopy constraint above intersects $\gamma_{eight}$ at most once. 

    More generally, the same property holds for every fixed-end minimiser of $\mathcal{M}_h$ joining $q_1\in U_{eight}$ to a point $q_2$ lying in one of the bounded connected components of $\R^2\setminus\gamma_{eight}$.
\end{proposition}
\begin{proof}
    The proof is a direct consequence of the intersection properties of minimisers. Let $\gamma\colon[0,+\infty)\to\R^2\setminus\{c_1,c_2\}$ be a collision minimiser and consider the eight-shaped minimiser $\gamma_{eight}$. For any $[a,b]\subset [0,+\infty)$, the restriction $\gamma|_{[a,b]}$ is a collision-less minimiser of $\mathcal{M}_h$ among all $H^1$ paths joining $\gamma(a)$ and $\gamma(b)$. Therefore, all intersection between $\gamma|_{[a,b]}$ and $\gamma_{eight}$ are transversal by Proposition \ref{prop:intersection_minimisers}. Two distinct transversal intersections thus determine a singular 2-gon, contradicting Proposition \ref{prop:intersection_minimisers}. Clearly, the same holds for any fixed-end minimiser joining $q_1\in U_{eight}$ to a point $q_2$ in one of the bounded connected components of $\R^2\setminus\gamma_{eight}$.
\end{proof}
At this point, we establish the existence of collision minimisers. 
\begin{proposition}[Collision minimisers]\label{prop:collision_minimisers}
    Let $h$ satisfy \eqref{hyp:energy}, $q\in\R^2\setminus\{c_1,c_2\}$ and $c$ be one of the two centres. Then, there exists a collision minimiser of $\mathcal{M}_h$ joining $q$ to $c$.
\end{proposition}
\begin{proof}
    Assume $c=c_1=0$ and let $q\in\R^2\setminus\{c_1,c_2\}$. Let $(q_n)$ be a sequence $q_n\to 0$. Let $(\gamma_n)\subset H^1([0,1])$ be a sequence of fixed-end minimisers of the Jacobi-Maupertuis length
    \[
        \mathcal{L}_h(\gamma_n) = \int_0^1\lvert\dot{\gamma}_n(t)\rvert\sqrt{V(\gamma_n(t))+h}\,dt
    \]
    such that $\gamma_n(0)=q$, $\gamma_n(1)=q_n$. Such minimisers exist and are collision-less for any $n$ (see Section \ref{sec:fixed_end}); moreover, thanks to Proposition \ref{prop:eight}, they keep at positive distance from $c_2$. Reparametrise each $\gamma_n$ using the Jacobi-Maupertuis arc-length parameter $s$, i.e., in such a way that $g_h(\dot{\gamma}_n(s), \dot{\gamma}_n(s)) = 1$. Keeping the same notation for $\gamma_n$, we obtain $T_n=\mathcal{L}_h(\gamma_n)>0$ with $\gamma_n\colon[0, T_n]\to\R^2$. 
    
    We first claim that $T_n\to+\infty$ as $n\to+\infty$. To see that, fix $r\in(0, |q|)$ so that $c_2\not\in \bar{B}_r$. Since $q_n\to 0$, we have that $\gamma_n$ intersects $\partial B_r$ for $n$ large enough; defining 
    \[
        \sigma_n = \sup\{s\in [0, T_n]:\, |\gamma_n(s)|=r\}
    \]
    we then see that $\gamma_n([\sigma_n, T_n])\subset B_r\setminus\{0\}$. Write $\rho_n(s)=|\gamma_n(s)|$ and, recalling that $V(x)\sim\frac{f_1(x)}{|x|^2}$ as $x\to 0$, we have 
    \[
        \mathcal{L}_h(\gamma_n|_{[\sigma_n, T_n]})\ge C_1\int_{\sigma_n}^{T_n}\frac{|\dot{\gamma}_n(s)|}{|\gamma_n(s)|}\,ds,
    \]
    for some constant $C_1>0$. Since $|\dot{\rho}_n(s)|\le|\dot{\gamma}_n(s)|$, we obtain 
    \[
        \begin{aligned}
            \mathcal{L}_h(\gamma_n|_{[\sigma_n, T_n]})\ge C_1 \int_{\sigma_n}^{T_n}\frac{|\dot{\rho}_n(s)|}{\rho_n(s)}\,ds\ge C_1\left\lvert \int_{\sigma_n}^{T_n}\frac{\dot{\rho}_n(s)}{\rho_n(s)}\,ds\right\rvert = C_1 \log(r/|q_n|).
        \end{aligned}
    \]
    Since $q_n\to 0$ as $n\to+\infty$, the claim is proved. 
    
    Next, we claim that, for a fixed window of time, the minimizers do not get too close to the center. Fix $S>0$. We claim that there exist $\ve = \ve(S)>0$ and $n\in\N$ such that
    \begin{equation}\label{eq:claim}
        |\gamma_n(s)|\ge\ve\quad\forall s\in[0,S],\ \forall\ n\ge N. 
    \end{equation}
    Assume by contradiction that the claim is false. Then, there exist a subsequence $(\gamma_{n_k})$, points $s_k\in[0,S]$ such that 
    \[
    |\gamma_{n_k}(s_k)|\to 0,\quad s_k\to s^*\in[0, S],\ \text{as}\ k\to+\infty. 
    \]
    If $r\in(0, |q|)$, for large $k$ we have $|\gamma_{n_k}(s_k)| < r$. We define 
    \[
        \zeta_k = \sup\left\lbrace s\in[0, s_k]:\, |\gamma_{n_k}(s)|=r\right\rbrace
    \]
    so that $\gamma_{n_k}([\zeta_k, s_k])\subset B_r\setminus\{0\}$. Writing again $\rho_k(s)=|\gamma_{n_k}(s)|$ and arguing exactly as in the previous claim, we obtain
    \[
        \mathcal{L}_h(\gamma_{n_k}|_{[\zeta_k,s_k]}) \ge C_2\log(r/|\gamma_{n_k}(s_k)|)
    \]
    for some constant $C_2>0$ and so $\mathcal{L}_h(\gamma_{n_k}|_{[\zeta_k, s_k]})\to+\infty$ as $k\to+\infty$. Since $\gamma_{n_k}$ is parametrised over arc-length, we have that
    \[
        s_k = \mathcal{L}_h(\gamma_{n_k}|_{[0, s_k]})\ge \mathcal{L}_h(\gamma_{n_k}|_{[\zeta_k, s_k]}),
    \]
    but this gives a contradiction since $s_k\to\bar{s}\le S$ and so claim \eqref{eq:claim} is proved.
    
    Recalling that $S>0$ is fixed, we can define the compact annulus 
    \[
        A_S = \{x\in\R^2:\, \ve\le|x|\le |q|\}
    \]
    and, from the second claim, we know that $\gamma_n([0, S])\subset A_S$ for any $n\ge N$. Since $A_S$ is compact and does not intersect the singularities, there exist two positive constants $m_S, M_S$ such that 
    \[
        m_S\le V(x)+ h\le M_S\quad\forall x\in A_S.
    \]
    Now, recalling that $g_h(\dot{\gamma}_n(s), \dot{\gamma}_n(s)) =1$, we have
    \[
        |\dot{\gamma}_n(s)| = \frac{1}{\sqrt{V(\gamma_n(s))+ h}}\le \frac{1}{\sqrt{m_S}},
    \]
    for a.e. $s\in[0,S]$, and so the family $(\gamma_n)$ is equi-Lipschitz on $[0, S]$. Moreover, integrating on $[0, S]$, we also deduce a uniform bound on $\gamma_n(s)$ in $[0, S]$. Therefore, by Ascoli-Arzela theorem we can extract a sub-sequence converging uniformly in $[0, S]$ to a continuous limit $\gamma^S\colon[0, S]\to A_S$. Taking a sequence $S_k\to+\infty$ and using a diagonal argument, we can extend the limit to a continuous curve $\gamma\colon[0, +\infty)\to\R^2\setminus\{c_1,c_2\}$ such that $\gamma_n\to\gamma$ uniformly on each compact interval $[0,S]$. 

    Now, we fix $0\le a < b<+\infty$ and check that $\gamma|_{[a,b]}$ is a minimiser of $\mathcal{L}_h$ among all $H^1$ paths from $\gamma(a)$ to $\gamma(b)$. Each $\gamma_n$ minimises $\mathcal{L}_h$ among $H^1$ paths with fixed ends at $\gamma_n(a)$, $\gamma_n(b)$. Now, let $\psi\in H^1([a,b])$ such that $\psi(a)=\gamma(a)$ and $\psi(b)=\gamma(b)$. Consider a sequence $(\psi_n)$ with endpoints $\gamma_n(a), \gamma_n(b)$ such that $\mathcal{L}_h(\psi_n)\to\mathcal{L}_h(\psi)$. From the minimality of $\gamma_n$ we have $\mathcal{L}_h(\gamma_n|_{[a,b]})\le\mathcal{L}_h(\psi_n)$. Using Fatou lemma and uniform convergence on $[a,b]$, we have
    \[
        \mathcal{L}_h(\gamma|_{[a,b]})\le \liminf\limits_{n\to+\infty}\mathcal{L}_h(\gamma_n|_{[a,b]})\le\liminf\limits_{n\to+\infty}\mathcal{L}_h(\psi_n) = \mathcal{L}(\psi),
    \]
    hence $\gamma|_{[a,b]}$ is a minimiser of the Jacobi-Maupertuis length. Since collision-less minimisers of $\mathcal{L}_h$ are reparametrisations of minimisers of the Maupertuis functional, $\gamma$ is a local minimizer between its endpoints.
    
    Note that minimizers must enter an Euclidean ball of fixed small radius in a uniform time. More precisely, for any $r$ small enough there exists $S_r$ such that $\gamma_n ([S_r ,T_n]) \subseteq B_r$. Indeed, if we assume by contradiction that the claim is false, there exist sequences $n_k$ and $S_k\to \infty$ for which $\gamma_{n_k}(S_k) \notin B_r$. This implies that there exists at least a point of $\partial B_r$ at infinite distance from $q$ which is a contradiction. 
    This easily implies that  $\lim_{s\to+\infty}\gamma(s) = 0$ and concludes the proof. 
\end{proof}

\subsection{Realising bi-infinite sequences}
Let $h$ satisfy \eqref{hyp:energy} and consider the energy shell
\[
    \mathcal{E}_h = \left\lbrace(q,v)\in (\R^2\setminus\{c_1,c_2\})\times\R^2:\, \frac12 |v|^2 - V(q) = h\right\rbrace
\]
and consider $\psi_1, \psi_2$ collision minimisers of $\mathcal{M}_h$ respectively joining $c_1$ to a point in $\partial\mathcal{D}$ and $c_2$ to another point in the same boundary, so that $\psi_1$ and $\psi_2$ do not intersect. 
Define the set of all initial data which lie in one of the collision minimisers
\[
    \mathcal{T} = \left\lbrace (q,v)\in\mathcal{E}_h:\, q\in\psi_j,\ \text{for some}\ j\in\{1,2\}\right\rbrace.
\]
The set $\mathcal{T}$ is the union of two open cylinders $C_1$ and $C_2$, one of which can be decomposed in two connected components $C_i^\pm$ after removing from $\mathcal{T}$ the initial data $(\psi_i(t), \pm\dot{\psi}_i(t))$. We keep the same notation for $\mathcal{T}$ and we observe that now $\langle v, \dot{\psi}_i(t)^\perp\rangle \neq 0$ for any initial velocity in $\mathcal{T}$ at intersection points. 

Denote by $\Phi_t\colon\mathcal{E}_h\to\mathcal{E}_h$ the flow associated with the motion equation, so that each initial data $(q,v)$ is mapped to $\Phi_t(q,v)$ after a time $t$. 
Let us introduce the subset of $\mathcal{T}$ of all initial data coming back to $\mathcal{T}$ in finite time:
\[
    \Sigma = \left\lbrace (q,v)\in\mathcal{T}:\, \inf\{t>0:\, \Phi_t(q,v)\in\mathcal{T}\}<+\infty\right\rbrace
\]
so that, for any $(q,v)\in\Sigma$, its first return time $T(q,v)$ is well-defined. Inductively, taking $\Sigma^1 = \Sigma$, we also define
\[
    \Sigma^n = \left\lbrace (q,v)\in\Sigma^{n-1}: \, \inf\{t>0:\, \Phi_t(q,v)\in\mathcal{T}\}<+\infty\right\rbrace,\ n\ge 1.
\]
Consider the map $\mathcal{J}(q,v) = (q, -v)$ and define the set of initial data whose associated trajectories come back infinitely times, both in the future and in the past, namely
\[
    \Sigma^\infty = \cap_{n\ge 1} \left(\Sigma^n \cap \mathcal{J}(\Sigma^n)\right),
\]
which is non empty as a consequence of Corollary \ref{cor:collisionless}. Define also a first return map on $\Sigma^\infty$ which acts in this way:
\begin{equation}
	\label{eq:def_Phi}
    \Phi\colon\Sigma^\infty\to\Sigma^\infty,\quad (q,v)\mapsto \Phi(q,v)= \Phi_{T(q,v)}(q,v).
\end{equation}
To define the itinerary of a point $(q,v)\in\Sigma^\infty$, we note that, initially, $q$ belongs to one of the collision curves $\psi_j$, and the same holds for the projection on the configuration space of $\Phi(q,v)$. Each trajectory with initial data in $\Sigma^\infty$ is defined on the whole $\R$ and so intersect transversally the curves $\psi_j$ infinitely many times. Therefore, we can associate to such trajectory a bi-infinite sequence
\[
    (s_i)_{i\in\Z},\quad s_i\in S=\{\pm 1, \pm 2\},
\]
which represents the sequence of signed transversal intersections with $\psi_1$ and $\psi_2$ (where the orientation is given by the generators $\al_1$ and $\al_2$, see Figure \ref{fig:generators}). Recall that the Bernoulli shift map $\sigma\colon S^\Z\to S^\Z$ such that $\sigma(s_i) = (s_{i+1})$ acts by shifting forward the whole sequence. In this way, we have naturally defined an equivariant map 
\begin{equation}\label{eq:projection_map}
    \pi\colon\Sigma^\infty\to S^\Z,\quad \text{such that}\ \pi\circ\Phi = \sigma\circ\pi. 
\end{equation}

Recalling the statement of Theorem \ref{thm:coding}, we are interested in sub-shift of finite type (see \cite[Chapter 2]{LiMa95}) including only those sequences $(s_i)_{i\in\Z}$ which satisfy these two conditions
\begin{enumerate}
    \item there are no repetitions of the same symbol of length $\ge k$ in the sequence
    \item there are no patterns of the form $(\pm 1, \mp 1)$ or $(\pm 2, \mp 2)$
    \item there exist $i\neq j$ such that $|s_i|= 1$ and $|s_j|= 2$
\end{enumerate}
which correspond to the minimisers of $\mathcal{M}_h$ found in Corollary \ref{cor:collisionless}. We denote by $X$ the closure, with respect to the metric of the Bernoulli shift, of the set of the sequences satisfying the two conditions above. The following result is well-known and a complete proof can be found in \cite{BarCan}. 

\begin{lemma}
    The set $(X, \sigma)$ is a chaotic sub-system of $(S^\Z, \sigma)$ and $X$ is homeomorphic to a Cantor set. 
\end{lemma}

As a first step, we have a preliminary lemma which shows that any periodic sequence in $X$ is realised by a collisionless periodic Maupertuis minimiser. The proof follows directly from Lemma \ref{lem:compactness_minimisers}.
\begin{lemma}\label{lem:periodic_sequences} 
    Fix $h$ satisfying \eqref{hyp:energy}, $k\in\N$ and consider a periodic $k$-admissible sequence $(s_i)_{i\in\Z}\in X$ of period $m\ge 2$, generated by a finite admissible sequence $(s_1,\ldots,s_m)\in\{\pm1,\pm2\}^m$. Then, there exists a collisionless periodic minimiser of the Maupertuis functional $\gamma\in\mathcal{H}(s_1,\ldots,s_m)$, whose itinerary coincides with $(s_i)_{i\in\Z}$. 

    Moreover, there exists a constant $C=C(h,k)>0$ independent on the sequence such that 
    \[
        \min_{t\in\cerchio}|\gamma(t)- c_j|\ge C,\quad \forall\,j=1,2.
    \]
\end{lemma}

Now we show that every sequence in $X$ is realised by a collisionless trajectory and thus prove Theorem \ref{thm:coding}.

\begin{proof}[Proof of Theorem \ref{thm:coding}]
    Fix $h$ satisfying \eqref{hyp:energy}, $k\in\N$ and let $(s_i)_{i\in\Z}$ be a bi-infinite $k$-admissible sequence in $X$. If the sequence is periodic, Lemma \ref{lem:periodic_sequences} guarantees the existence of a minimiser intersecting $\psi_1, \psi_2$ with the order given by the sequence and thus realising it. Note that, since $\psi_1$ and $\psi_2$ are defined as limits of minimisers homotopic to non-intersecting line segments and each minimiser given by Corollary \ref{cor:collisionless} realises the minimum number of intersections in its homotopy class, this number coincides with the period of the sequence and the signs are uniquely determined by the homotopy class and the choice of the orientations. 

    Fix now a non-periodic sequence $(s_i)_{i\in\Z}\in X$. Thanks to Lemma \ref{lem:periodic_sequences}, we can consider a sequence of periodic minimisers $(\gamma_k)$ of period $2T_k>0$ of $\mathcal{M}_h$ approximating the sequence. This means that the itineraries of $(\gamma_k)$ converge locally to the bi-infinite sequence $(s_i)_{i\in\Z}$, i.e., for any $r\in\N$, there exists $k(r)$ such that, for any $k\ge k(r)$, the itinerary of $\gamma_k$ contains the finite block 
    \[
        x_r = (s_{-r}, s_{-r+1}, \ldots, s_{r-1}, s_r).
    \]
    For any $r\in\N$, one can define the limiting minimal time necessary for $(\gamma_k)$ to realise $x_r$ as
    \[
        a_r = \lim_{k\to+\infty} \inf_{T>0}\{T:\,\ \text{the itinerary of }\gamma_k|_{[-T,T]}\ \text{contains}\ x_r\}.
    \]
    With the same argument as in \cite[Lemma 6.11]{BarCan}, one can prove that 
    \begin{itemize}
        \item $\lim_{k\to+\infty}T_k\to+\infty$
        \item $a_r<a_{r+1}<+\infty$ for any $r\in\N$
        \item $\lim_{r\to+\infty} a_r = +\infty$
    \end{itemize}
    and that there exist a sequence $\ve_n\to 0$ and a subsequence $(\gamma_{k_n})$ of $(\gamma_k)$ such that the restrictions $\eta_n = \gamma_{k_n}|_{[-a_r - \ve_n, a_r + \ve_n]}$ satisfy
    \begin{itemize}
        \item $\eta_n(a_r + \ve_n)\in\psi_{s_r}$ and $\eta_n(a_r - \ve_n)\in\psi_{s_{|-r|}}$
        \item each $\eta_n$ is a minimiser of $\mathcal{M}_h$ between its endpoints
        \item the itinerary of $\eta_n$ contains the block $x_r$
        \item the sequence $(\eta_n)$ is bounded in $H^1$ and admits a uniformly convergent subsequence
        \item the uniform limit $\eta^r\colon[-a_r, a_r]\to\mathcal{D}$ is itself a minimiser of $\mathcal{M}_h$.
    \end{itemize}
    Now, assume by contradiction that $\eta^r$ has a collision at time $t_0\in(-a_r, a_r)$ with $c_j$. Then, the uniform convergence $\eta_n\to\eta^r$ forces $\eta_n$ to enter arbitrarily small neighbourhoods of $c_j$ for $n$ large. On the other hand, each $\eta_n$ is a restriction of a periodic minimiser which realises the finite admissible block $x_r$. Hence, by Lemma \ref{lem:periodic_sequences} we reach a contradiction and $\eta^r$ is collisionless. Since all curves remain in a compact subset of $\mathcal{D}\setminus\{c_1,c_2\}$, standard regularity arguments imply that $\eta_n\to\eta^r$ in $C^2$ on $[-a_r, a_r]$. 
    
    Iterating this procedure and refining further the subsequence, one can obtain a limiting curve $\eta^{r+1}$ defined on a larger interval $[-a_{r+1}, a_{r+1}]$ which contains the symbolic block $x_{r+1}$ and whose restriction to $[-a_r, a_r]$ coincides with $\eta^r$ and thus contains the itinerary $x_r$. Since $a_r\to+\infty$, a diagonal argument yields a sequence which converges $C^2$ on arbitrarly large compact intervals to a limit $\gamma\colon\R\to\mathcal{D}\setminus\{c_1,c_2\}$ and the proof is concluded.

\end{proof}

The image of the map $\pi$ defined in \eqref{eq:projection_map} thus contains the set $X$. We conclude the section showing that this implies the existence of a compact invariant set $\Sigma_X$ for the map $\Phi$ defined in \eqref{eq:def_Phi}.
\begin{corollary}
    The map $\pi$ defined in \eqref{eq:projection_map} is continuous and its image contains $X$. Therefore, there exists a compact invariant subset $\Sigma_X\subset\Sigma^\infty$ on which $(\Sigma_X, \Phi)$ is topologically semi-conjugated to the chaotic system $(X, \sigma)$. 
\end{corollary}
\begin{proof}
 The continuity of the map $\pi$ is a straightforward consequence of the continuous dependence of initial data. For the existence of the invariant set, we can reason as follows.
 Defined the set $A\subset \mathcal{T}$ as the set of initial conditions of all the minimal trajectories realizing an itinerary in $X\subseteq \mathcal{S}^\Z$. We set $\Sigma_X = \bar{A}$ and show that $\Sigma_X$ is invariant. 
 
 First of all note that if $(q_n,v_n)\in \mathcal{T}$ are initial conditions in $A$, then there exists $T$ such that
 \[
 \sup_n\inf_{t>0} \left\{t:\,\Phi_t(q_n,v_n)\in \mathcal{T}\right \}\le T.
 \]
 So, if we assume $(q_n,v_n)\to (\bar q,\bar v)\in \bar A$, by compactness there is a subsequence $n_k$ for which $T_{n_k} \to \bar T$ and by continuity of the flow it holds that $\Phi_{T_n}(q_n,v_n)\to \Phi_{\bar T}(\bar q, \bar v)\in \bar A$. Thus, we have shown that for any $(q,v)\in \bar{A}$, $\Phi(q,v)\in \bar A$ as well and so $\Sigma_X$ is invariant. 
\end{proof}
\begin{remark}
	If the Jacobi-Maupertuis metric $g_h(v,v) = (h+V(x))|v|^2$ at energy $h$ has  negative curvature $\kappa_J$, one can prove that the above semi-conjugation is actually a conjugation as an application of the Gauss-Bonnet theorem. Precise details can be found in \cite[Theorem 6.7, Lemma 6.12]{BarCan}. 
	
	It should be noted that for the non-perturbed ($\mathcal{R}=0$) relativistic two center problem this computation is straightforward.  Having to deal with just superpositions of homogenous potentials, the curvature is indeed negative. From standard formulas for conformal changes of the metric it holds that
	\begin{align*}
	\kappa_J = \frac{-\Delta \log(h+V)}{h+V}& =-2m c^2\frac{\Delta \log((h_{\mathrm{rel}}+V_{\mathrm{rel}})^2-m^2c^4)}{(h_{\mathrm{rel}}+V_{\mathrm{rel}})^2-m^2c^4} \\&=-2m c^2\frac{\Delta \log(h_{\mathrm{rel}}-mc^2+V_{\mathrm{rel}})+\Delta \log(h_{\mathrm{rel}}+mc^2+V_{\mathrm{rel}})}{(h_{\mathrm{rel}}+V_{\mathrm{rel}})^2-m^2c^4} 
	\end{align*}
	where we used that the effective potential $h + V = ((V_\mathrm{rel}+h_\mathrm{rel})^2-(mc^2)^2)/2mc^2$.
	If we are able to give a sign to terms of the form $\Delta \log (h' +V_\mathrm{rel})$ for $h'>0$ we are then able to give a sign to the whole curvature. For the specific choice of $V_\mathrm{rel}$ we have
	\[
	-\Delta \log(h'+V_{\mathrm{rel}})=\frac{\vert\nabla V_\mathrm{rel}\vert^2-(h'+V_{\mathrm{rel}})\Delta V_{\mathrm{rel}}}{(h'+V_{\mathrm{rel}})^2}
	\]
	The laplacian is easily computed in polar coordinates and the follow formula holds
	\[
	(h'+V_{\mathrm{rel}})\Delta V_{\mathrm{rel}}= \left(h'+\frac{\mu_1}{\vert q-c_1\vert}+\frac{\mu_2}{\vert q-c_2\vert}\right)\left(\frac{\mu_1}{\vert q-c_1\vert^3}+\frac{\mu_2}{\vert q-c_2\vert^3}\right)
	\]
	and for the gradient part on has
	\[
	\vert\nabla V_\mathrm{rel}\vert^2=\frac{\mu_1^2}{\vert q-c_1\vert^4}+\frac{\mu_2^2}{\vert q-c_2\vert^4} + \frac{\mu_1\mu_2\langle q-c_1,q-c_2\rangle }{\vert q-c_1\vert^3\vert q-c_2\vert^3}
	\]
	Thus, after some cancellations, we have that
	\[
	-(h'+V_{\mathrm{rel}})^2\Delta \log(h'+V_{\mathrm{rel}})= - h' \left(\frac{\mu_1}{\vert q-c_1\vert}+\frac{\mu_2}{\vert q-c_2\vert}\right) - \frac{\mu_1\mu_2}{\vert q-c_1\vert \vert q-c_2\vert}R(q)
	\]
	where $R(q)= \frac{1}{\vert q-c_1\vert^2}+\frac{1}{\vert q-c_2\vert^2}-\frac{\langle q-c_1,q-c_2\rangle }{\vert q-c_1\vert^2\vert q-c_1\vert^2}$. It is however clear that by Cauchy-Schwarz inequality
	\[
	R(q)\ge \frac{1}{\vert q-c_1\vert^2}+\frac{1}{\vert q-c_2\vert^2}-\frac{1 }{\vert q-c_1\vert\vert q-c_2\vert} \ge \frac{1}{\vert q-c_1\vert\vert q-c_2\vert}>0
	\]
\end{remark}

\section{Scattering and trapped trajectories}\label{sec:scattering}

In this final section we prove Theorems \ref{thm:scattering_bi} and \ref{thm:trapped}. The aim is to construct entire solutions of the motion and energy equations with prescribed symbolic behaviour and asymptotic directions at infinity. 

More precisely, scattering solutions of Theorem \ref{thm:scattering_bi} will be obtained as limits of fixed-end minimisers whose endpoints diverge to infinity along assigned directions. Using the compactness properties established in the previous sections, we will prove that such minimisers converge, up to subsequences, to entire collisionless trajectories connecting two asymptotic directions.

Similarly, trapping solutions of Theorem \ref{thm:trapped} will be constructed by combining the compactness arguments used in the proof of Theorem \ref{thm:coding} with the asymptotic analysis of scattering trajectories. In this way, we obtain entire solutions which are asymptotic to infinity in one time direction and realise a prescribed positively infinite symbolic itinerary in the other.

\subsection{Proof of Theorem \ref{thm:scattering_bi}}

We begin by establishing uniform control on trajectories starting sufficiently far from the interaction region $\mathcal{D}$. In particular, we show that solutions of the motion and energy equations with large initial radius behave asymptotically like free scattering trajectories: they are globally defined, remain far from the origin, and admit well-defined asymptotic direction and velocity. The following lemma provides uniform quantitative estimates on this behaviour and will be used repeatedly in the compactness argument for sequences of minimisers with diverging endpoints.
\begin{lemma}\label{lem:flow}
    There exist $K, a_0, C_1, C_2 >0$ such that, for any solution $\eta$ of the motion equation \eqref{eq:motion} at energy $h$ satisfying \eqref{hyp:energy}, with $|\eta(0)| = \rho > K$ and $\langle \dot{\eta}, \eta \rangle \ge 0$, we have
    \begin{enumerate}[label=\roman*)]
        \item $\eta$ is defined on the whole $[0, +\infty)$ and $|\eta(t)|\ge \sqrt{a_0 t^2 + \rho^2}$ for any $t\in[0, +\infty)$
        \item there exists $\lim\limits_{t\to+\infty} \frac{\eta(t)}{|\eta(t)|} = e^{i\vt_\infty}$, for some $\vt_\infty\in\cerchio$ 
        \item there exists $\lim\limits_{t\to+\infty} \dot{\eta}(t) = v_\infty$, with $v_\infty = \sqrt{2h}e^{i\vt_\infty}$ and $\left\lvert v_\infty - \dot{\eta}(0)\right\lvert\le \frac{C_2}{\rho^{\beta -1}}$.
    \end{enumerate}
\end{lemma}
\begin{proof}
    To prove $i)$, from the motion and energy equation we have 
    \[
        \frac{d^2}{dt^2}\left(\frac{|\eta(t)|^2}{2}\right) = \langle\eta(t), \ddot{\eta}(t)\rangle + |\dot{\eta}(t)|^2 = \langle\eta(t), \nabla V(\eta(t))\rangle + 2(h + V(\eta(t)))
    \]
    and so, from the assumption at infinity \eqref{hyp:infinity}, there exist $K>0$ and $a_0, a_1>0$ such that, as long as $|\eta(t)|>K$ we have
    \[
        a_1 \ge \frac{d^2}{dt^2}\left(\frac{|\eta(t)|^2}{2}\right)\ge a_0.
    \]
    In particular, the function $|\eta(t)|^2/2$ is convex and, if $\langle \dot{\eta}(0), \eta(0)\rangle >0$, it is also increasing and defined for all times $t\ge 0$. A comparison principle implies that 
    \[
        |\eta(t)|^2\ge |\eta(0)|^2 + a_0 t^2 + \langle\dot{\eta}(0), \eta(0)\rangle t \ge \rho^2 + a_0 t^2.
    \]
    Now, to prove $ii)$ we use the motion equation. We call $\frac{\eta(t)}{|\eta(t)|}= e^{i\vt(t)}$ and we compute
    \[
        \frac{d}{dt} e^{i\vt(t)} = i\dot{\vt}(t) e^{i\vt(t)} = \frac{\dot{\eta}(t)}{|\eta(t)|} - \frac{1}{|\eta(t)|} \langle \dot{\eta}(t), e^{i\vt(t)}\rangle e^{i\vt(t)}
    \]
    so that
    \[
        \dot{\vt}(t) = \left\langle \frac{d}{dt} e^{i\vt(t)}, \frac{J\eta(t)}{|\eta(t)|}\right\rangle = \frac{1}{|\eta(t)|^2}\left\langle\dot{\eta}(t), J\eta(t)\right\rangle. 
    \]
    Moreover, we have
    \[
        \frac{d}{dt}\left(\langle\dot{\eta}(t), J\eta(t)\rangle\right) = \langle\dot{\eta}(t), J\dot{\eta}(t)\rangle + \langle\ddot{\eta}(t), J\eta(t)\rangle = \langle\nabla V(\eta(t)), J\eta(t)\rangle. 
    \]
    From \eqref{hyp:infinity}, we then deduce that the function $\psi(t) = \langle\dot{\eta}(t), J\eta(t)\rangle$ has a limit $\psi(+\infty)$ as $t\to+\infty$. In particular, from the fundamental theorem of calculus and point $i)$, we deduce
    \[
        |\psi(+\infty) - \psi(t)| \le \int_t^{+\infty} \frac{C}{\rho^\beta (a_0 (\tau/\rho)^2 + 1 )^{\beta/2}}\,d\tau,\ \forall\, t\ge 0. 
    \]
    Using the change of variable $d\tau = \rho ds$, we see that there exists $c_3>0$ such that
    \[
        |\psi(+\infty) - \psi(t)| \le \frac{1}{\rho^{\beta-1}}\int_0^{+\infty} \frac{C}{(a_0 s^2 + 1)^{\beta/2}}\,ds \le \frac{c_3}{\rho^{\beta-1}},
    \]
    and so, in particular, $\psi\in L^1(0,+\infty)$. Now, write $\vt(t) = \vt(0) + \int_0^t \dot{\vt}(\tau)\, d\tau$ and we have just seen that $\dot{\vt}=\frac{\psi(t)}{|\eta(t)|^2}\in L^1(0, +\infty)$, so there exists $\vt_\infty=\lim_{t\to+\infty}\vt(t)$.  Thus, assertion $ii)$ follows.  However, more can be said, indeed
    \[
        |\vt_\infty - \vt(0)| \le \int_0^{+\infty}\frac{|\psi(\tau)|}{|\eta(\tau)|^2}\,d\tau \le \int_0^{+\infty} \frac{|\psi(0)| + \frac{c_3}{\rho^{\beta-1}}}{\rho^2(a_0 (\tau/\rho)^2 + 1)}\,d\tau
    \]
    with the same change of variable as above, we have
    \[
        |\vt_\infty - \vt(0)| \le \frac{1}{\rho}\int_0^{+\infty} \frac{|\psi(0)| + \frac{c_3}{\rho^{\beta-1}}}{a_0 s^2 + 1}\,ds.
    \]
     Thus, noting that $\vert\psi(0)\vert \le \rho$ and so $\vert \vt(0) - \vt_\infty \vert \le \frac{\pi}{2 \sqrt{a_0}}  $ as $\rho\to+\infty$. 

    Finally, to prove $iii)$, we use again the motion and energy equations. First of all, since $\ddot{\eta} = \nabla V(\eta)$, from \eqref{hyp:infinity} and $i)$ we deduce that
    \[
        |\ddot{\eta} (t)|\le \frac{C}{|\eta(t)|^\beta} \le \frac{C}{(a_0 t^2 + \rho^2)^{\beta/2}},\quad\forall\,t\ge 0
    \]
    and so $\ddot{\eta}\in L^1(0, +\infty)$. Therefore, from the fundamental theorem of calculus we deduce the existence of $v_\infty = \lim_{t\to+\infty} \dot{\eta}(t)$. In particular, using again the change of variable $ds = \rho d\tau$, we have that 
    \[
    \begin{aligned}
        |v_\infty - \dot{\eta}(0)| \le \int_0^{+\infty}|\ddot{\eta}(\tau)|\,d\tau \le C \int_0^{+\infty} \frac{d\tau}{\rho^{\beta}(a_0 (\tau/\rho)^2 + 1)^{\beta/2}} &= \frac{C}{\rho^{\beta-1}} \int_0^{+\infty}\frac{ds}{(a_0 s^2 + 1)^{\beta/2}}  \\
        & = \frac{c_3}{\rho^{\beta-1}}
    \end{aligned}.
    \]
    Now, assume by contradiction that $v_\infty$ is not parallel to $e^{i\vt_\infty}$ and so that $|\langle v_\infty, J e^{i\vt_\infty}\rangle| \ge\ve$ for some $\ve>0$. From $ii)$ and what we have just proved, we can assume also that, for large $t$, $\lvert\langle\dot{\eta}(t), Je^{i\vt(t)}\rangle \lvert\ge \ve$. As in the proof of $ii)$, we write $\frac{\eta(t)}{|\eta(t)|} = e^{i\vt(t)}$ and we compute
    \[
        \dot{\vt}(t) = \frac{1}{|\eta(t)|}\langle \dot{\eta}(t), \frac{J\eta(t)}{|\eta(t)|}\rangle \ge \frac{\ve}{|\eta(t)|}\ge \frac{\ve}{\sqrt{a_0 t^2 + \rho^2}},
    \]
    the last inequality coming from point $i)$. However, in the proof of $ii)$ we have seen that $\dot{\vt}\in L^1(0, +\infty)$, which leads to a contradiction since the last term of the inequality is not integrable on $[0, +\infty)$. 

\end{proof}

At this point, we are ready to give the proof of Theorem \ref{thm:scattering_bi}.

\begin{proof}[Proof of Theorem \ref{thm:scattering_bi}]
    Fix $h$ satisfying \eqref{hyp:energy}, two asymptotic directions $\vt^-, \vt^+\in[0, 2\pi)$, $n\ge 2$ and a finite $(n+1)$-admissible sequence $s_1,\ldots,s_n$. Let $R_0>0$ be such that the compactness and fixed-end results of Section \ref{sec:fixed_end} apply uniformly for endpoints outside $B_{R_0}$. For every $R>R_0$ we consider the line $\ell^-$ tangent to the ball of radius $R$ at the point $Re^{i\vt^-}$ and the line $\ell^+$ tangent to the ball of radius $R$ at the point $Re^{i\vt^+}$.  For any couple of points $q^\pm \in \ell^\pm$ we consider a minimisers $\gamma_R^{q^-,q^+}$ with prescribed itinerary $(s_1,\ldots, s_n)$ and define $\gamma_R$ as a minimizer satisfying
    	\[
    	\mathcal{L}_{h}(\gamma_R) \le \mathcal{L}_h (\gamma_R^{q^-,q^+}), \quad \forall q^\pm \in \ell^\pm.
    	\]
    In other words, $\gamma_R$ is a minimizers between paths joining the lines $\ell^\pm$ and prescribed itinerary $(s_1,\ldots, s_n)$.
    
   	We parametrise each minimiser by physical time, so that 
    \[
        \frac12|\dot{\gamma}_R(t)|^2 - V(\gamma_R(t)) = h.
    \]
    Let $K > \max\{|c_1|, |c_2|\} + 1$. Since each minimiser $\gamma_R$ has itinerary $(s_1,\ldots,s_m)$, it necessarily enters the compact region containing the centres and thus it intersects $\partial B_K$. By time-translation invariance of the motion equation, we can assume that $|\gamma_R(0)|=K$.  Since the lines $\ell^\pm$ lie outside the ball of radius $R$, the minimizer $\gamma_R$ eventually enters $B_R$. We now prove that the time spent inside the ball $B_R$ diverges as $R\to+\infty$. Let $t_R^+>0$ be such that 
    \[
        |\gamma_R(t_R^+)|=R,\quad |\gamma_R(t)|<R,\ \forall\,t\in[0,t_R^+).
    \]
    We observe that the Jacobi-length of $\gamma_R$ connecting $\partial B_R$ and $\partial B_K$ goes to $+\infty$ as $R\to+\infty$. Indeed, if we write $r(t) = |\gamma_R(t)|$ we have    
    \[
    \begin{aligned}
        \mathcal{L}_h(\gamma_R|_{[0, t_R^+]})=\int_{0}^{t_R^+} |\dot{\gamma}_R(t)|\sqrt{V(\gamma_R(t))+h}\,dt&\ge C_0\int_{0}^{t_R^+} |\dot{r}(t)| \,dt, \\
        &\ge C_0 |r(t_R^+)- r(0)| = C_0 |R-K|.
    \end{aligned}
    \]
    since $V(x)+h$ is uniformly bounded from below by a constant $C_0>0$.
   Since $\gamma_R(t)$ is parametrised by physical time it satisfies the energy equation 
    \[
        |\dot{\gamma}_R(t)|^2 = 2(h+ V(\gamma_R(t))).
    \]
    Moreover, thanks to Lemma \ref{lem:compactness_minimisers}, the minimisers $\gamma_R$ stay away from the centers, so there exists a constant $C_1>0$ such that 
    \[
    C_0 |R- K|\le \mathcal{L}_h(\gamma_R|_{[0, t_R^+]}) = \int_0^{t_R^+} \sqrt{2}(h + V(\gamma_R(t)))\,dt\le C_1 t_R^+.
    \]
    Hence, the time spent inside the ball $B_R$ goes to infinity as $R\to+\infty$. The same argument works if applied backward in time. 
    
    Therefore, for any $n>0$, the restriction $\gamma_R|_{[-n,n]}$ is well-defined for $R$ large enough. 
    Using Lemma \ref{lem:compactness_minimisers}, we know that all $\gamma_R$ stay at uniformly bounded distance from the centres. Thus, from the energy equation, we know that 
    \[
        |\dot{\gamma}_R|^2 = 2h + 2V(\gamma_R) \le 2h + C 
    \]
    and so the family $\gamma_R$ is pre-compact in the $C^0([-n,n])$ topology, for any $n\in\N$. 
    Now, from the motion equation in $[-n,n]$, we get uniform bounds on the derivatives $\dddot{\gamma}_R$, which implies the pre-compactness in $C^2([-n,n])$. With a straightforward application of the Ascoli-Arzela theorem and a diagonal argument, we get the $C_{loc}^2$ convergence of the sequence $\gamma_R$ to a limit $\gamma_\infty\colon\R\to\R^2$.
    
    We now show that the family $\gamma_R$ satisfies the assumption of Lemma \ref{lem:flow}.  We focus on the $+\infty$-case the proof of the analogous property at $-\infty$ goes through the same argument.  Each minimiser $\gamma_R$ has to cross the ball $\partial B_K$ thanks to the topological constraint we imposed. Let $t^+_R\to+\infty$ such that $|\gamma_R( t^+_R)| = R$. A Lagrange-Jacobi type inequality as in the proof of Lemma \ref{lem:flow} shows that $\langle\gamma_R(t^+_R), \dot{\gamma}_R(t^+_R)\rangle >0$. Note in addition that, by construction, the curve $\gamma_R$ is orthogonal to $\ell^+$ when it intersects it. This means that there exists a time $T^+_R\ge t^+_R$ for which $\dot{\gamma}_R(T^+_R)$ is parallel to $e^{i \theta^+}$. In particular, from point $i)$ of Lemma \ref{lem:flow}, these minimisers are defined in the whole $[0,+\infty)$ and satisfy (here the Landau notation refers to the limit as $R\to \infty$)
    \begin{equation}\label{eq:profile}
    	\lim_{t\to+\infty}\dot{\gamma}_R(t) = \dot{\gamma}_R(+\infty)= \sqrt{2h} e^{i\vt^+} + O(R^{1-\beta}), \quad 
    	\lim_{t\to+\infty} \frac{\gamma_R(t)}{|\gamma_R(t)|} = \gamma_R(+\infty) =  e^{i\vt^+} + o(1).
    \end{equation}
	
	Now we show that $\dot{\gamma}_\infty(+\infty) = \sqrt{2h}e^{i\vt^+}$.  We can write
    \[
        \dot{\gamma}_\infty(t)= \dot{\gamma}_\infty(t)- \dot{\gamma}_R(t)+ \dot{\gamma}_R(t)- \dot{\gamma}_R(+\infty)+ \dot{\gamma}_R(+\infty)
    \]
    Take an interval $(-n,n)$ large enough so that $|\gamma_\infty(t)|=1/\ve$ for some $t\in(0,n)$ and any $\ve>0$. Then, uniform $C^2$ convergence implies that, for all $R>R(\ve)$ 
    \[
        \left\lvert \dot{\gamma}_\infty(t)- \dot{\gamma}_R(t) \right\rvert\le \ve\quad\text{and}\quad \left\lvert \lvert \gamma_\infty(t)\rvert - \lvert \gamma_R(t)\rvert \right\rvert \le \ve.
    \]
    Therefore, we have that 
    \[
        \left\lvert \dot{\gamma}_\infty(t)-  \sqrt{2h}e^{i\vt^+}\right\rvert \le \vert  \dot{\gamma}_\infty(t)- \dot{\gamma}_R(t)\vert+ \vert \dot{\gamma}_R(t)- \dot{\gamma}_R(+\infty) \vert+ \left\lvert \dot{\gamma}_R(+\infty)  - \sqrt{2h} e^{i\vt^+}\right\rvert\le \ve +O(\ve^{\beta-1})
    \]
    since $\ve$ is aribitraty and the limit velocity exists, we get the claim. We conclude, thanks to point $iii)$ of Lemma \ref{lem:flow}, that $\lim_{t\to\infty} \gamma_\infty(t)/\vert \gamma_\infty(t)\vert = e^{i\vt^+}$.

    The uniform convergence on compact sets, together with Lemma \ref{lem:compactness_minimisers}, implies that the limit $\gamma_\infty$ has itinerary coded by the sequence $s_1,\ldots, s_m$ and the proof is concluded. 
\end{proof}

\subsection{Proof of Theorem \ref{thm:trapped}}
We now prove the existence of trapping trajectories with prescribed positively infinite itineraries. The argument combines the two approximation procedures developed in the previous sections. On the one hand, we use the scattering construction of Theorem \ref{thm:scattering_bi} in the past, prescribing an asymptotic direction at $-\infty$ and letting one endpoint diverge to infinity. On the other hand, we approximate the prescribed positively infinite symbolic sequence by longer and longer finite admissible blocks, exactly as in the proof of Theorem \ref{thm:coding}. More precisely, for each finite truncation of the itinerary we consider a minimiser with one endpoint constrained to the tangent line to the sphere at the point $Re^{i\vartheta^-}$ and the other fixed inside the compact region $\mathcal D$. The compactness results established in Sections \ref{sec:compactness} and \ref{sec:fixed_end} imply that these minimisers converge, up to subsequences, in the $C^2_{\mathrm{loc}}$ topology to a global solution of \eqref{eq:motion}--\eqref{eq:energy}. The limiting trajectory inherits the prescribed asymptotic direction in the past and realises the whole positively infinite itinerary.

\begin{proof}[Proof of Theorem \ref{thm:trapped}]
    Fix $h$ satisfying \eqref{hyp:energy}, $\vt^-\in[0,2\pi)$, $k\in\N$ and a $k$-admissible positively infinite sequence $(s_i)_{i\in\N}\subset \{\pm 1, \pm 2\}^\N$. Consider the point $q=(0, K)$, $R>0$ and the family of approximating sequences of length $m\in\N$, $\sigma_m = (s_i)_{i=1}^m$.  Let $\gamma_R^m$ be a family of fixed-end minimisers connecting the tangent line to the sphere at the point $Re^{i\vartheta^-}$ and $q$, and realising the sequence $\sigma_m$. Reparametrise each $\gamma_R^m$ so that $|\gamma_R^m(0)| = K$ and 0 is the first instant of time spent inside the ball $B_K$ for any $\gamma_R^m$. We are going to pass to the limit as $R\to+\infty$ and $m\to+\infty$; therefore, we may assume $R=m$ and we simply call the sequence of minimisers $(\gamma_m)$. With the same proof of Theorem \ref{thm:scattering_bi}, one shows that the family $\gamma_m$ converges $C_{loc}^2$ on compact sets to a limit $\gamma_\infty\colon\R\to\R^2$. From Lemma \ref{lem:boundedness_minimisers}, we know that 
    \[
    \begin{cases}
        |\gamma_\infty(t)|>K \quad\text{for all } t<0 \\
        |\gamma_\infty(t)|<K \quad\text{for all } t>0
    \end{cases}
    \]
    and for this reason we can repeat the proof of Theorem \ref{thm:scattering_bi} to show that $\gamma_\infty(t)/|\gamma_\infty(t)|\to e^{i\vt^-}$ as $t\to-\infty$. Similarly, one sees that $\gamma_\infty$ agrees with the infinite sequence $(s_i)_{i\in\N}$ using the same argument as in the proof of Theorem \ref{thm:coding}, this time working only with positively infinite sequences. 
\end{proof}

\section*{Data availability}
No datasets were generated or analysed during the current study.

\section*{Acknowledgements}
The authors are grateful to Alberto Boscaggin for many valuable discussions on the topics addressed in this paper.

\section*{Funding}
The authors acknowledge the support of the INdAM--GNAMPA Research Group. The first author also acknowledges financial support from the Alexander von Humboldt Foundation.

\bibliography{references}
\bibliographystyle{plain}

\medskip

\noindent
S. Baranzini\\
 Universit\`a  San Raffaele di Roma\\
Via di Val Cannuta 247, 00166 Roma, Italy\\
\vspace{-0.2cm}

\noindent
Ruhr-Universit\"at Bochum\\
Universit\"atsstra\ss e 150, 44801 Bochum S\"ud, Germany \\
\texttt{stefano.baranzini@uniroma5.it}

\vspace{0.8cm}

\noindent
G. M. Canneori \\
Dipartimento di Matematica ``Giuseppe Peano'', Universit\`a degli Studi di Torino\\
Via Carlo Alberto 10, 10123 Torino, Italy\\
\texttt{gianmarco.canneori@unito.it}

\end{document}